\documentclass[english,reqno]{amsart}

\usepackage{latexsym}
\usepackage{amssymb,amsbsy,amsmath,amsfonts,amssymb,amscd}
\usepackage{mathrsfs}

\usepackage{multirow}

\usepackage[english]{babel}
\usepackage{algorithm2e}
\usepackage{algpseudocode}

\usepackage{subcaption}
 \usepackage{graphicx}
\usepackage{epstopdf}

\usepackage{color}
\usepackage{array}

\newcolumntype{L}{>{\centering\arraybackslash}m{3cm}}

\usepackage{cases}
\theoremstyle{plain}

\newtheorem{theorem}{Theorem}
 
\newtheorem{lemma}{Lemma}

\theoremstyle{definition} % For roman text in the body

\newtheorem{remark}{Remark}

\newcommand{\Amat}{\mathsf{A}}

\newcommand{\VV}{\mathbb{V}}
\newcommand{\EE}{\mathbb{E}}
\newcommand{\RR}{\mathbb{R}}
\newcommand{\Var}{\mathrm{Var}}
\newcommand{\opL}{\mathcal{L}}
\newcommand{\opA}{\mathcal{A}}
\newcommand{\calJ}{\mathcal{J}}
\newcommand{\Cov}{\text{Cov}}

\newcommand{\rd}{\mathrm{d}}

\newcommand{\Sp}{\mathbb{S}}

\graphicspath{{../graph/}}

\title{Online learning in optical tomography: a stochastic approach}
\author{Ke Chen} 
\address{Mathematics Department, University of Wisconsin-Madison, 480 Lincoln Dr., Madison, WI 53705 USA.}
\email{ke@math.wisc.edu}
\author{Qin Li} 
\address{Mathematics Department and Wisconsin Institutes for Discovery, University of Wisconsin-Madison, 480 Lincoln Dr., Madison, WI 53705 USA.}
\email{qinli@math.wisc.edu}
\author{Jian-Guo Liu} 
\address{Department of Mathematics and Department of Physics, Duke University, Durham, NC 27708 USA.}
\email{jliu@phy.duke.edu}

\thanks{The work of K.C. and Q. L. is supported in part by a start-up fund of Q.L. from UW-Madison and National Science Foundation under the grant DMS-1619778. The work of J.L. is supported in part by National Science Foundation under the grant DMS-1514826 and DMS-1107444: RNMS KI-Net.}

\begin{document}
\maketitle

\begin{abstract}
We study the inverse problem of radiative transfer equation (RTE) using stochastic gradient descent method (SGD) in this paper. Mathematically, optical tomography amounts to recovering the optical parameters in RTE using the incoming-outgoing pair of light intensity. We formulate it as a PDE-constraint optimization problem, where the mismatch of computed and measured outgoing data is minimized with same initial data and RTE constraint. The memory and computation cost it requires, however, is typically prohibitive, especially in high dimensional space. Smart iterative solvers that only use partial information in each step is called for thereafter. Stochastic gradient descent method is an online learning algorithm that randomly selects data for minimizing the mismatch. It requires minimum memory and computation, and advances fast, therefore perfectly serves the purpose. In this paper we formulate the problem, in both nonlinear and its linearized setting, apply SGD algorithm and analyze the convergence performance.
\end{abstract}

%%%%%%%%%%%%%%%%%%%%%%%%%%%%%%%%%%%%%%%%%%%%%%%%%%
\section{Introduction}
Optical tomography is a form of computed tomography that extracts tomographic images of objects to be studied using information of light transmitted and scattered through it. It has been vastly used in many applications: in medical imaging near infrared light (NIR) is sent into biological tissues for tumor or bone structure~\cite{Hielscher_rheumatoid1,Hielscher_rheumatoid2}; in outer space studies: during Galileo's travel around Jupiter, pictures are taken by the near infrared mapping spectrometer (NIMS), and scientists recover components of atmosphere on each satellite~\cite{Doute_mapping_so2}. Typically scientists inject a certain amount of light into a bulk of material, and measure the outgoing light intensities at the boundaries. By collecting many such incoming and outgoing light intensity pairs, scientists infer for the optical information of the material.

Mathematically, light is typically characterized by the radiative transfer equation (RTE). It characterizes photon particles that scatter and get absorbed in materials with various optical properties. Optical tomography, therefore is formulated as the inverse problem of the radiative transfer equation. The equation reads:

\begin{equation}\label{eqn:RTE}
v\cdot\nabla_x f + \sigma(x)f = \int_\VV k(x, v,v')f(x, v')\rd{v'}\,,
\end{equation}
where $f(x,v)$, defined on phase space, is the distribution of particles at location $x$ with velocity $v$. Here $x \in \Omega\subset\RR^d$ with $d = 2,3$, and $v \in \VV = \Sp^{d-1}$, the unit sphere in $\RR^d$. $k(x,v,v')$ is the scattering coefficient and it shows the probability of particles moving in direction $v'$ changing to direction $v$ at location $x$, and $\sigma(x)$ is the total absorption coefficient that represents certain amount of photon particles being absorbed by the material. The equation has a unique solution with the following boundary condition:
\begin{equation}\label{eqn:Inflow}
f|_{\Gamma_-} = \phi(x,v)\,,
\end{equation}
where $\Gamma_-$ collects the coordinates on $\partial\Omega$ with incoming velocities (and $\Gamma_+$ collects the outgoings):

\begin{equation*}
\Gamma_\pm =\{(x,v): x\in\partial\Omega, \pm v\cdot n_x >0\}\,.
\end{equation*}

Here $n_x$ stands for the normal direction pointing out of the domain at point $x\in\partial\Omega$. The wellposedness of the equation in the general $L_p$ space has been studied in~\cite{Egger_Schlottbom_Lp}. Define the albedo operator that maps the incoming boundary condition to the outgoing data:

\begin{equation}
\opA: \phi(x,v) \to \left(n_x\cdot v\right) f(x,v)|_{\Gamma_+}\,.
\end{equation}
In the forward problem setting, the optical properties $\sigma$ and $k$ are known and one computes $\left(n_x\cdot v\right) f|_{\Gamma_+}$ for arbitrarily given $\phi$. In the inverse setting, one obtains all possible $(\phi, \left(n_x\cdot v\right) f|_{\Gamma_+})$ pairs and uses them ($\opA$ information) to recover $\sigma$ and $k$.Note there are multiple ways to define $\opA$ depending on the measurements. For example, $\opA$ could map $\phi$ to $f|_{\Gamma_+}$ or the angular averaged measurements $\int (n_x\cdot v)f \rd{v}$.

The problem, due to its large application, has been extensively studied from many aspects. On the analytical side people concern the wellposedness and the stability. More precisely, we ask: 1. does $\opA$ contain enough information to extract all coefficients; 2. how sensitive the recovery is towards the measurements. The first question was initially addressed by a pioneer paper in~\cite{ChoulliStefenov}, in which the authors used the singular decomposition technique to prove the uniqueness of the recovery in 3D if $\sigma$ has no $v$ dependence. This technique was later extended to study angular average data~\cite{Bal09,Bal08} and the case where $\sigma$ has the $v$ dependence~\cite{Stefenov2}. The second question was looked at as early as in~\cite{Wang}, and the bad conditioning was addressed by increasing the modulation frequency in the time-harmonic case~\cite{BalMonard_time_harmonic}. In~\cite{ChenLiWang}, the authors studied the stability's dependence on the Knudsen number and recover, to some extent, the ill-conditioning of the Calder\'on type problem in the diffusion limit. See~\cite{Bal_review} for a review.

On the computation side, different application setups provide different types of measurements, and it drives the development of various numerical techniques~\cite{Arridge_piecewise_const,Arridge_couple,Ren_fPAT,TangHanHan,Egger2015,Prieto_Dorn,Cheng_Gamba_Ren_Doping}. A very general descriptions are found in influential books~\cite{Epstein_MedicalImaging,Natterer_CT}. Generally speaking people regard it as an optimization problem with PDE constraints. More precisely, one tries to minimize the mismatch between the measurements and the numerical results assuming the RTE is satisfied. In this process, $L_2$, $L_1$ or TV norm of the coefficients are added as penalties to fit certain a prior knowledge. The biggest challenge here, of course, is the size of the problem: in every iteration a forward solver is called, and this deals with the distribution function $f$ that lives on phase space and has $N^5$ degrees of freedom in 3D (assuming each direction takes $N$ points). Some techniques have been applied to reduce the cost. This includes using the linearization as an approximation~\cite{Ren_review}, applying gradient-based instead of the Jacobian~\cite{Arridge_gradient} etc. An early review was given by Arridge and Ren~\cite{Arridge99,Ren_review}.

None of the algorithms, however, is online. With traditional approaches, one typically assumes that many experiments are done, and a large number of pairs of $(\phi,(n_x\cdot v) f|_{\Gamma_+})$ are collected ahead of time. These data points are stored and used all-together in the computation as a whole batch. An immediate disadvantage is the run-time memory and computational cost: in each iteration, all experiment measurements are called for to adjust the parameter. We develop online algorithms for inverting RTE in this paper. In particular, we apply the stochastic gradient-descent method. It is a standard online algorithm: we start with one data point (one incoming-outgoing pair), and gradually adjust by incorporating new ones randomly selected from the data pool. This way, in each iteration, only very few data points are required, significantly accelerating the optimization. We stop once error tolerance is achieved. This online routine minimally requires data points, and avoids experiment waste. As will be shown later, numerically it is drastically more efficient too. We have to mention that we are not the very first group to explore the possibility of incorporating the random sampling techniques to inverse problems. The randomized version of the Kaczmarz's method (originally extensively studied in~\cite{Haltmeier07}) was proposed in~\cite{Leitao_Svaiter} for elliptic equations with Dirichlet-to-Neumann map as the data.

In the following, we review the stochastic gradient descent method in section 2, and show the formulation of the inverse problem in both the linearized and the original nonlinear setting in section 3. Section 4 collects our numerical experiments.

\section{Stochastic gradient descent method}\label{sec:SGD}
We briefly review the stochastic gradient descent method in a general setting. The notation is consistent within this section, and will be adjusted accordingly in later sections.

Stochastic gradient descent (SGD) algorithm and many of its variants are often used to solve optimization problems of the form
\begin{equation}
\min \calJ(\sigma) =\frac{1}{N}\sum_{k=1}^N \calJ_k(\sigma)\,,
\end{equation}
where $\calJ$ is average of all $\calJ_k$, which maps the trainable parameters $\sigma\in\RR^d$ to $\RR$. $N$ is the training sample size and could be very large depending on applications. To solve the problem using the standard gradient descent method, one updates $\sigma_n$ for each step, the parameter at $n$-th step, using:
\begin{equation}
\sigma_{n+1} = \sigma_n - \eta\nabla_x \calJ(\sigma_n) = \sigma_n - \frac{\eta}{N}\sum_{k=1}^N\nabla_\sigma \calJ_k(\sigma_n)\,.
\end{equation}
Here $\eta$ is the gradient descent time step, or sometimes termed learning rate. This method requires derivative with respect to $\sigma$ for all $\calJ_k$ evaluated at $\sigma_n$ and the computation could be prohibitively expensive for big $N$.

SGD method is a stochastic alternative of gradient descent method (GD). It replaces the full gradient $\nabla_\sigma \calJ$ by only one sampled version in each iteration. In its simplest form, the SGD iteration is written as
\begin{equation}\label{eqn:SGD}
\sigma_{n+1} = \sigma_n - \eta_n\nabla_\sigma \calJ_{\gamma_n} (\sigma_n)\,,
\end{equation}
where $\eta_n$ is still the learning rate which may or may not vary in $n$. The learning direction is no longer the gradient of the whole cost function but is replaced by that of one sample $\calJ_{\gamma_n}$ randomly chosen from the sample pool ($\{\gamma_n\}$ is a random variable evenly chosen from $\{1, 2, \cdots , N\}$). Per iteration, SGD requires only one sample's derivative in $\sigma$ at $\sigma_n$. Since the computational complexity is much reduced compared with GD, SGD is of favor for many large scale problems~\cite{Bottou2010,Bottou2012}.

There are many works addressing the performance of SGD. Studies were done on quantifying the convergence rate, choosing optimal learning rate, checking condition number dependence, and extending to nonconvex objectives. Many different variants (large batch training, stochastic average gradient, problem in the linear setting, and semi-stochastic method etc.)~\cite{RechtWright_SGD,Roux_exp,Zhang_cond,Mandt_conv,Strohmer2008,Richtarik_sscd} have been studied too for various of purposes. The convergence in the most general setting is still unknown, and several techniques have been employed to explain it~\cite{Bottou2010,Moulines,Needell}. Among them we specifically mention the technique that links SGD algorithm with stochastic partial differential equations (SDEs). The computation of SDE itself also attracts some studies~\cite{Ozen_Bal}.

In fact, if one rewrites SGD as:
\begin{equation}\label{eqn:rewrite}
\sigma_{n+1} - \sigma_n = -\eta\nabla_\sigma \calJ(\sigma_n) + \eta\nabla_\sigma(\calJ-\calJ_{\gamma_n})(\sigma_n)\,,
\end{equation}
with $\eta$ independent on $n$, it could be explained as the discretization for the following SDE:
\begin{equation}\label{eqn:SDE}
\rd{X_t} = b(X_t)\rd t + a(X_t)\rd W_t\,,
\end{equation}
with $\eta$ being the time step, $b(\sigma) = -\nabla_\sigma \calJ(\sigma)$ being the drift, and $a(x) = \left(\eta \Sigma\right)^{1/2}$ is the Brownian motion with the covariance defined by:
\begin{equation}\label{eqn:def_sigma}
\Sigma = \frac{1}{N}\sum_k \left(\nabla \calJ(\sigma) - \nabla \calJ_k(\sigma)\right)\left(\nabla \calJ(\sigma) - \nabla \calJ_k(\sigma)\right)^\top\,.
\end{equation}

This observation was made rigorous in~\cite{LiTaiE_SGD}, and we cite the theorem here:
\begin{theorem}
Let $T > 0$ and define $\Sigma$ as in~\eqref{eqn:def_sigma}. Assume $\calJ$, $\calJ_k$ are Lipschitz continuous, have at most linear asymptotic growth and have sufficiently high derivatives. Then,
the stochastic process $X_t$ with $t\in[0,T]$ satisfying
\begin{equation}
\rd X_t = -\nabla \calJ(X_t)\rd t + (\eta \Sigma(X_t))^{1/2} \rd W_t
\end{equation}
is an order 1 weak approximation of the SGD, meaning: for every $g$ of polynomial growth, there exists $C > 0$, independent of $\eta$, such that for all $n = 0,1,...,n_T = T/\eta$,
\begin{equation}
|\EE g(X_{n\eta}) - \EE g(\sigma_n)| < C\eta\,.
\end{equation}
Here $X_{n\eta}$ is the solution to the SDE~\eqref{eqn:SDE} evaluated at $n\eta$ and $\sigma_n$ is the $n$-th iteration solution to the SGD algorithm~\eqref{eqn:SGD}.
\end{theorem}

Consider the connection between SDE and the Fokker-Planck equation, the rewrite of the scheme~\eqref{eqn:rewrite} can also be regarded as the discretization for:
\begin{equation}
\partial_tu = b(x)\cdot\nabla u + \frac{1}{2}\eta\Sigma:\nabla^2u\,.
\end{equation}
and this was made rigorous in~\cite{FengLiLiu_SGD} by using a small jump approximation in Markov process.

These results essentially claim that the SGD results can be interpreted by the solution to the SDE and the Fokker-Planck. Once the connection is drawn, the analysis to the SDE could be carried to understand the convergence behavior of SGD. Indeed, the equation contains a drift term and a diffusion term, in charge of bringing two types of behaviors. Suppose the initial guess is far away from the optimal and $\nabla_\sigma \calJ$ is very big, then the drift term will dominate. The solution therefore will firstly move according to the direction given by the drift term and quickly converge to a state to have $\nabla_\sigma \calJ =0$. Once the drift term is small enough, the diffusion term will dominate, and this gives a Brownian motion like oscillating behavior. The two phases are termed the descent phase and the fluctuations phase, and the transition time is usually determined by setting $\EE(X_t) = \sqrt{\Var(X_t)}$.

The solution to the SDE could be made more explicit when $\eta$, the learning rate is small. In the zero limit of $\eta$, the diffusion term shrinks. By performing the standard asymptotic expansion in $\eta$ to~\eqref{eqn:SDE}, the solution to the SDE, in the leading order, becomes:
\begin{equation}
X_t \sim\mathcal{N}(X_{0,t},\eta S_t)\,,
\end{equation}
a Gaussian process centers at $X_{0,t}$, a deterministic process that satisfies:
\begin{equation*}
\frac{\rd}{\rd t}X_{0,t} = -\nabla \calJ(X_{0,t})\,,
\end{equation*}
with fluctuation $S_t$ governed by:
\begin{equation}
\frac{\rd}{\rd t}S_t = -S_tH_t - H_tS_t +\Sigma_t\,.
\end{equation}
Here $H_t = \nabla^2\calJ(X_{0,t})$ is the Hessian of $\calJ$ evaluated at $X_{0,t}$, and $\Sigma_t = \Sigma(X_{0,t})$, with $\Sigma$ defined in~\eqref{eqn:def_sigma}. The interested readers are referred to~\cite{LiTaiE_SGD} for more details.

\section{Inverting for optical properties of RTE}
We apply SGD to the inverse problem in RTE. We first unify the notation. We focus on the critical case in this paper, meaning the absorption and the scattering term have the same intensities. The method takes minimum changes when the two terms are different. The calculation will be presented in Remark~\ref{rmk:absorption} and numerical experiments will be demonstrated in Section 4. The equation writes, in 2D:
\begin{equation*}
\begin{cases}
v\cdot\nabla f = \sigma(x_1,x_2)\opL[f]\,,\quad x=(x_1,x_2)\in[0,1]^2, v\in\mathbb{S}\\
f|_{\Gamma_-} = \phi(x_1,x_2,v)
\end{cases}\,,
\end{equation*}
where $\opL[f]$ is the collision term:
\begin{equation*}
\opL[f] = \int_{\mathbb{S}}f\rd{v} - f = \langle f\rangle_v - f\,.
\end{equation*}
Here $\rd{v}$ is a normalized measure. If we write $v=(\cos\theta,\sin\theta)$ then:
\begin{equation}\label{eqn:rte_2d}
\begin{cases}
\cos\theta\partial_{x_1} f + \sin\theta\partial_{x_2} f = \sigma(x_1,x_2)\opL[f]\,,\quad (x_1,x_2,\theta)\in[0,1]^2\times[-\pi,\pi]\\
f|_{\Gamma_-} = \phi(x_1,x_2,\theta)
\end{cases}\,.
\end{equation}
In the equation $\Gamma_-$ collects coordinates on the four boundary lines with velocities pointing into the domain:
\begin{equation*}
\begin{aligned}
\Gamma_- =&  ~\{x_1=0,x_2\in[0,1],\cos\theta>0\}\cup \{x_1=1,x_2\in[0,1],\cos\theta<0\} \\
				   &\cup \{x_1\in[0,1],x_2=0,\sin\theta>0\} \cup\{x_1\in[0,1],x_2=1,\sin\theta<0\}\,,
\end{aligned}
\end{equation*}
and $\Gamma_+$ collects the rest.

%The study is done on the discrete domain both experimentally and numerically. In each spatial direction we divide the domain into $N_x$ even pieces, and the velocity domain has $N_\theta$ grid points. The discrete domain then becomes:
%\begin{equation*}
%\Omega\times \Sp^1 = \{[x_i,y_j,\theta_k]: x_i = ih\,, y_j = jh\,, \theta_k = kh_\theta\}\,,\quad\text{with}\quad h = 1/N_x\,, h_\theta = 2\pi/N_\theta\,.
%\end{equation*}
%The light is injected and measured at discrete grids. On each boundary line, $N_\theta/2$ directions are the incoming direction, and the other $N_\theta/2$ directions are outgoing. The discrete incoming boundary writes as (we only show $\Gamma^\text{d}_-$):
%\begin{equation*}
%\Gamma^\text{d}_- = \{x_0,y_i (\forall i),\theta_j (\cos\theta_j>0)\}\cup \{x_{N_x},y_i (\forall i),\theta_j (\cos\theta_j<0)\}\cup\{x_i(\forall i),y_0,\theta_j (\sin\theta_j>0)\}\cup\{x_i(\forall i),y_{N_x},\theta_j (\sin\theta_j<0)\}.
%\end{equation*}
For every run of the experiment, one turns on light supported on $\Gamma_-$ with prescribed intensities, termed $\phi^{(k)}$ and collects outgoing intensities, termed $\psi^{(k)}$. We note that $\psi^{(k)}$ contains pollution in the measuring procedure. The superindex $k$ labels the round of experiment.

Throughout the section we may encounter the following norms:
\begin{equation*}
\|f\|^2_\pm = \int_{\Gamma_\pm} |f|^2\rd{x}\rd{v}\,,\quad \|f\|^2_2 = \int_{\Omega\times\mathbb{S}}|f|^2\rd{x}\rd{v}\,.
\end{equation*}

The following two subsections are devoted to nonlinear and linearized versions of the inverse problem, both of which employ dual problems for extracting information.

\subsection{Nonlinear version}
We look for the scattering coefficient $\sigma(x_1,x_2)$ in the nonlinear setting in this section. This is achieved by matching the result of the albedo operator acting on the incoming data $\phi^{(k)}$ and the measured data $\psi^{(k)}$. More precisely we perform the PDE-constraint optimization. Define the cost function:
\begin{equation}\label{eqn:cost}
\calJ_k = \frac{1}{2}\|(n\cdot v) f^{(k)} - \psi^{(k)}\|^2_+ + \frac{\alpha}{2}\|\sigma\|^2_2\,.
\end{equation}
and the PDE constraint:
\begin{equation}\label{eqn:forward}
(v\cdot\nabla -\sigma\mathcal{L})f^{(k)} =0\,,\quad f^{(k)}|_{\Gamma_-} = \phi^{(k)}\,,
\end{equation}
then we minimize:
\begin{equation}\label{eqn:opt_nonlinear}
\begin{cases}
\min_{\sigma} \quad&\frac{1}{N}\sum_k \calJ^{(k)} = \frac{1}{N}\sum_k\left(\frac{1}{2}\|(n\cdot v)f^{(k)} - \psi^{(k)}\|^2_+ + \frac{1}{2}\alpha\|\sigma\|^2_2\right)\\
\text{s.t.}\quad &v\cdot\nabla f^{(k)} = \sigma(x_1,x_2)\opL[f^{(k)}]\,,\quad f^{(k)}|_{\Gamma_-} = \phi^{(k)}
\end{cases}\,.
\end{equation} 
%Here $\calJ$ is the cost function defined by:
%\begin{equation*}
%\calJ = \frac{1}{2}\|f|_{\Gamma_+} - \psi\|^2_- + \frac{\alpha}{2}\|\sigma\|^2_2\,.
%\end{equation*}
A more compact form of the problem writes:
\begin{equation}
\min_{\sigma} \quad \frac{1}{N}\sum_k\left(\frac{1}{2}\|\opA(\sigma)[\phi^{(k)}] - \psi^{(k)}\|^2_+ + \frac{1}{2}\alpha\|\sigma\|^2_2\right)\,.% = \sum_k e_k(\sigma) + \frac{1}{2}\alpha\|\sigma\|_2
\end{equation} 
where $\opA$ is the albedo operator determined by $\sigma$ that maps the incoming data $\phi$ to the outgoing data $(n\cdot v)f|_{\Gamma_+}$ with $f$ satisfying~\eqref{eqn:forward}. A Kolmogorov regularizer $\|\sigma\|_2$ is added. Both the mismatch term and regularization term are measured in $L_2$ norm. Note that the data is of the form of $(n\cdot v)f|_{\Gamma_+}$ but not $f|_{\Gamma_+}$.

The update formula given by SGD is straightforward:
\begin{equation}\label{eqn:update}
\sigma_{n+1} = \sigma_n - \eta_n\frac{\rd }{\rd\sigma}\calJ_{\gamma_n}(\sigma_n)\,,
\end{equation}
with $\gamma_n$ randomly selected from $\{1\,,\cdots,N\}$. This means in each iteration, to update $\sigma$ from time step $n$ to $n+1$, one randomly select a incoming-outgoing pair $(\phi^{(\gamma_n)},\psi^{(\gamma_n)})$ and use the corresponding Fr\'echet derivative $\frac{\rd }{\rd\sigma}\calJ_{\gamma_n}$ evaluated at the previous data $\sigma_n$. To compute the Fr\'echet derivative, however, we need to employ the dual problem. We now derive it, and ignore sub-index $\gamma_n$ for conciseness of the notation.

We use the Lagrangian formulation. For all independent $f$, $\sigma$ and the duals $g$ and $\lambda$, we define the Lagrangian:
\begin{equation}
\mathsf{L}(\sigma,f,g,\lambda) = \calJ(\sigma,f) + \langle g\,,(v\cdot\nabla_x-\sigma\opL)f\rangle_2 +\langle\lambda\,, f|_{\Gamma_-}-\phi\rangle_-\,,
\end{equation}
with the last two terms coming from multiplying the two constraints (the equation and the boundary condition) by the Lagrangian multiplier $(g,\lambda)$. If the two constraints in~\eqref{eqn:forward} are satisfied, $f$ and $\sigma$ are no longer independent, and the last two terms disappear. On this special manifold, the Lagrangian is equivalent to $\calJ$. We denote such $f$ by $f_\sigma$. On $f = f_\sigma$ manifold:
\begin{equation}
\mathsf{L}(\sigma,f_\sigma,g,\lambda) = \calJ(\sigma,f_\sigma)\,.
\end{equation}
Take derivative with respect to $\sigma$:
\begin{equation*}
\frac{\rd\calJ}{\rd\sigma} = \frac{\partial\mathsf{L}}{\partial\sigma}+  \frac{\partial\mathsf{L}}{\partial f}\frac{\partial f}{\partial\sigma}\,.
\end{equation*}
Suppose $g$ and $\lambda$ are selected properly to make $\frac{\partial\mathsf{L}}{\partial f} = 0$, then:
\begin{equation}\label{eqn:Frechet}
\frac{\rd\calJ}{\rd\sigma} = \frac{\partial\mathsf{L}}{\partial\sigma} = \frac{\partial\calJ}{\partial\sigma} - \int_{\Sp} g\opL[f]\rd{v} = \alpha\sigma -\int_{\Sp} g\opL[f]\rd{v}\,,
\end{equation}
a formulation that could be explicitly computed.

To have $\frac{\partial\mathsf{L}}{\partial f} = 0$, we note that
\begin{align*}
\frac{\partial\mathsf{L}}{\partial f} &= \frac{\partial \calJ}{\partial f} + \frac{\partial}{\partial f}\langle g\,,(v\cdot\nabla_x-\sigma\opL) f\rangle_{\Omega\times\mathbb{S}} + \frac{\partial}{\partial f}\langle \lambda\,, f|_{\Gamma_-}-\phi\rangle_-\\
&=\frac{\partial}{\partial f}\left[ \frac{1}{2}\langle (v\cdot n) f-\psi\,,(v\cdot n) f-\psi\rangle_+ + \langle g\,,(v\cdot\nabla_x-\sigma\opL) f\rangle_{\Omega\times\mathbb{S}} + \langle \lambda\,, f|_{\Gamma_-}-\phi\rangle_-\right]\\
&=\frac{\partial}{\partial f}\left[ \frac{1}{2}\langle (v\cdot n) f-\psi\,,(v\cdot n) f-\psi\rangle_+ + \langle (v\cdot n)g\,,f\rangle_+ + \langle (-v\cdot\nabla_x-\sigma\opL) g\,,f\rangle_{\Omega\times\mathbb{S}} + \langle \lambda\,, f|_{\Gamma_-}-\phi\rangle_- + \langle (v\cdot n)g\,,f\rangle_-\right]
\end{align*}
where in the last equation we have used:
\begin{equation*}
\langle g\,,(v\cdot\nabla_x-\sigma\opL) f\rangle_{\Omega\times\mathbb{S}} =  \langle (-v\cdot\nabla_x-\sigma\opL)g\,, f\rangle_{\Omega\times\mathbb{S}} + \int_{\Gamma_+\cup\Gamma_-}(n\cdot v)fg\rd{x}\rd{v}\,.
\end{equation*}

We combine terms supported in different domains, and let them vanish:
\begin{align}\label{eqn:lagrangian_constraints}
\begin{cases}
(-v\cdot\nabla_x-\sigma\opL)g = 0\,,\quad &(x,v)\in\Omega\times\mathbb{S}\\
(n\cdot v)f - \psi + g = 0\,, \quad &(x,v)\in\Gamma_+\\
\lambda + (n\cdot v)g = 0\,,\quad &(x,v)\in\Gamma_-
\end{cases}\,.
\end{align}

The first two equations combined provide the restriction of $g$, i.e. $g$ satisfies the dual problem:
\begin{equation}\label{eqn:dual}
\begin{cases}
-v\cdot\nabla g = \sigma\opL[g]\\
g|_{\Gamma_+} = -(n\cdot v)f|_{\Gamma_+} + \psi
\end{cases}\,.
\end{equation}

In each iteration, to update~\eqref{eqn:update}, we compute~\eqref{eqn:dual} with the current guess $\sigma_n$ for $g$ using the mismatch being the boundary condition, and then generate the Fr\'echet derivative using~\eqref{eqn:Frechet}. We summarize the procedure in Algorithm~\ref{alg:nonlinear}.
\RestyleAlgo{boxruled}
\begin{algorithm}
\SetAlgoLined
\KwData{$N$ experiments with \begin{itemize}\item[1.] incoming data $\{\phi^{(k)}\}$;\item[2.] outgoing measurements: $\{\psi^{(k)}\}$;\item[3.] error tolerance $\varepsilon$;\item[4.] initial guess $\sigma_0$.\end{itemize}}
\KwResult{The minimizer $\sigma$ to the optimization problem~\eqref{eqn:opt_nonlinear} that is within $\varepsilon$ accuracy in residue.}
\While{$\|\frac{\rd}{\rd \sigma}\calJ_{\gamma_{n}}(\sigma_n)\|>\varepsilon$}{
Step I: randomly pick $\gamma_n\in\{1\,,\cdots,N\}$;\\
Step II: compute the forward problem~\eqref{eqn:forward} using boundary $\phi = \phi^{(\gamma_n)}$ with $\sigma = \sigma_n$ for $f^{(\gamma_n)}$;\\
Step III: compute the dual problem~\eqref{eqn:dual} using boundary $-(v\cdot n)f^{(\gamma_n)}|_{\Gamma_+} + \psi^{(\gamma_n)}$ with $\sigma = \sigma_n$ for $g^{(\gamma_n)}$;\\
Step IV: compute the Fr\'echet derivative~\eqref{eqn:Frechet}: $\frac{\rd}{\rd\sigma}\calJ_{\gamma_n}(\sigma_n) = \alpha\sigma_n - \int_{\Sp^1} \opL[f^{(\gamma_n)}]g^{(\gamma_n)}\rd{v}$;\\
Step V: update using~\eqref{eqn:update}: $\sigma_{n+1} = \sigma_n - \eta\frac{\rd}{\rd\sigma}\calJ_{\gamma_n}(\sigma_n)$.\\
$n=n+1$.
}
\caption{Find solution to the minimization problem~\eqref{eqn:opt_nonlinear}}\label{alg:nonlinear}
\end{algorithm}

We emphasize that for clinic interests, $N$ data points $\{\phi^{(k)},\psi^{(k)}\}$ do not need to be prepared beforehand. Before converging, in each step, an NIR laser is randomly placed on $\Gamma_-$ to generate $\phi^{(k)}$ and recerivers are placed on $\Gamma_+$ to collect $\psi^{(k)}$. Experiments are stopped once the algorithm gives convergence.  In this way, no redundant information is collected and this online algorithm maximally saves the experimenting time.

\begin{remark}\label{rmk:absorption}
It is of clinical interests that sometimes the equation~\eqref{eqn:rte_2d} is not in the critical case and the total absorption term is different from the scattering case. For simplicity we set the scattering being $1$ and study here how to recover the absorption term. The equation writes
\begin{equation}\label{eqn:rte_ab}
v\cdot \nabla_xf = \mathcal{L}f -\sigma f
\end{equation}
with boundary condition
\begin{equation*}
f|_{\Gamma_-} = \phi\,.
\end{equation*}
And the goal is to use the information of $\opA$ to recover $\sigma$. The minimization form writes as:
\begin{equation*}
\begin{cases}
\min_{\sigma} \quad&\frac{1}{N}\sum_k \calJ_k = \frac{1}{N}\sum_k\left(\frac{1}{2}\|(n\cdot v)f^{(k)} - \psi^{(k)}\|^2_+ + \frac{1}{2}\alpha\|\sigma\|^2_2\right)\\
\text{s.t.}\quad &v\cdot\nabla f^{(k)} =\opL[f^{(k)}] -\sigma(x_1,x_2)f^{(k)}\,,\quad f^{(k)}|_{\Gamma_-} = \phi^{(k)}
\end{cases}\,.
\end{equation*}
Following the same procedure, for all $k$, the Lagrangian is defined:
\begin{equation*}
\mathsf{L}(\sigma,f,g,\lambda) = \calJ(\sigma,f) + \langle g\,,(v\cdot\nabla_x-\opL+\sigma)f\rangle_2 +\langle\lambda\,, f|_{\Gamma_-}-\phi\rangle_-\,,
\end{equation*}
with the last two terms coming from the Lagrangian multiplier $(g,\lambda)$. On $f=f_\sigma$ manifold, the two terms drop and the Lagrangian is equivalent to $\calJ$, and:
\begin{equation}
\mathsf{L}(\sigma,f_\sigma,g,\lambda) = \calJ(\sigma,f_\sigma)\,.
\end{equation}
Take derivative with respect to $\sigma$:
\begin{equation*}
\frac{\rd\calJ}{\rd\sigma} = \frac{\partial\mathsf{L}}{\partial\sigma}+  \frac{\partial\mathsf{L}}{\partial f}\frac{\partial f}{\partial\sigma}=\alpha\sigma +\int_{\Sp} gf\rd{v}\,.
\end{equation*}
In the second equation we purposely select $g$ and $\lambda$ to have $\frac{\partial\mathsf{L}}{\partial f} = 0$. This requires:
\begin{align*}
\begin{cases}
(-v\cdot\nabla_x-\opL)g +\sigma g= 0\,,\quad &(x,v)\in\Omega\times\mathbb{S}\\
(n\cdot v)f - \psi + g = 0\,, \quad &(x,v)\in\Gamma_+\\
\lambda + (n\cdot v)g = 0\,,\quad &(x,v)\in\Gamma_-
\end{cases}\,.
\end{align*}
Once again the first two equations combined provide the restriction of $g$, the dual equation:
\begin{equation}\label{eqn:rte_ab_dual}
\begin{cases}
-v\cdot\nabla g = \opL[g] - \sigma g\\
g|_{\Gamma_+} = -(n\cdot v)f|_{\Gamma_+} + \psi
\end{cases}\,.
\end{equation}
In conclusion, to use SGD, we use the following in each iteration:
\begin{equation*}
\sigma_{n+1} = \sigma_n - \eta_n \frac{\rd}{\rd\sigma}\calJ_{\gamma_n} = \sigma_n -\eta_n\left(\alpha\sigma_n +\int_{\Sp} gf\rd{v}\right)\,,
\end{equation*}
where $f$ solves~\eqref{eqn:rte_ab} with $\phi^{(\gamma_n)}$ being the boundary condition and $\sigma_n$ being the media, and $g$ solves~\eqref{eqn:rte_ab_dual} with $\psi^{(\gamma_n)}$ and $\sigma_n$.
\end{remark}

\subsection{Linearized procedure}
In this section we describe the SGD applied on the linearized problem. The linearization is conducted upon $\sigma_0$, a background scattering coefficient believed to be very close to the true $\sigma$. The equation reads:
\begin{equation}\label{eqn:f_original}
\begin{cases}
v \cdot \nabla_x f&=\sigma\mathcal{L}f\,, \qquad (x, v) \in \Omega \times \mathbb{S}\,,  \\
f|_{\Gamma_-}&=\phi
\end{cases}\,,
\end{equation}
and its linearization is conducted assuming:
\begin{equation*}\label{eqn:sigma_tilde}
\tilde{\sigma}(x) = \sigma(x) - \sigma_{0}(x)\,\quad\text{and}\quad |\tilde{\sigma}|\ll |\sigma|\;\quad(a.e.)\,.
\end{equation*}
Then the linearized problem with the same inflow boundary condition reads as
\begin{equation}\label{eqn:f0}
\begin{cases}
v\cdot \nabla_x f_0=\sigma_0\mathcal{L}f_0\,, \qquad (x, v) \in \Omega \times \mathbb{S}\,,  \\
f_0|_{\Gamma_-}=\phi
\end{cases}\,.
\end{equation}
Let
\begin{equation*}
\tilde{f}(x,v) = f(x,v) - f_0(x,v)
\end{equation*}
be the fluctuation, we subtract the two equations~\eqref{eqn:f_original} and~\eqref{eqn:f0} for:
\begin{equation}\label{eqn:fluc_linear}
\begin{cases}
v\cdot \nabla_x \tilde{f}=\sigma_0\mathcal{L}\tilde{f}+\tilde{\sigma}\mathcal{L}f_0 \\
\tilde{f}|_{\Gamma_-}=0
\end{cases}\,,
\end{equation}
where we have omitted the higher order term $\tilde{\sigma}\mathcal{L}\tilde{f}$. To extract information to match the given data, we once again use the dual problem. Suppose we would like to find the information at $(x_\ast,v_\ast)\in\Gamma_+$, then we assign a delta function at the point for $g$ to use as the boundary condition:
\begin{equation}\label{eqn:dual_linear}
\begin{cases}
-v\cdot \nabla_x g=\sigma_0\mathcal{L}g\\
g|_{\Gamma_+}=\delta_{x_\ast,v_\ast}(x,v)
\end{cases}\,.
\end{equation}
Multiply~\eqref{eqn:dual_linear} by $\tilde{f}$ and multiply~\eqref{eqn:fluc_linear} by $g$ and subtract them, we get
\begin{equation}
(n_\ast\cdot v_{x_\ast}) \tilde{f}(x_\ast,v_\ast) = \int_\Omega \tilde{\sigma} \int_{\Sp^1} \mathcal{L}[f_0]g\rd{v}\rd{x}\,,
\end{equation}
Note the left hand side is known since:
\begin{equation}
(n_\ast\cdot v_{x_\ast})\tilde{f}(x_\ast,v_\ast) = (n_\ast\cdot v_{x_\ast})f(x_\ast,v_\ast)-(n_\ast\cdot v_{x_\ast})f_0(x_\ast,v_\ast)
\end{equation}
with the first term being a measurement $\psi(x_\ast,v_\ast)$, and the second computed from~\eqref{eqn:f0}. We denote it by:
\begin{equation} \label{eqn:b}
b(x_\ast,v_\ast;\phi):= (v_\ast\cdot n_{x_\ast}) \tilde{f}(x_\ast,v_\ast) =  \psi(x_\ast,v_\ast) - (v_\ast\cdot n_{x_\ast})f_0(x_\ast,v_\ast;\phi)\,,
\end{equation}
with $f_0$ implicitly depend on the inflow $\phi$. We also denote the Fredholm kernel on the right hand side:
\begin{equation} \label{eqn:beta}
\beta(x,x_\ast,v_\ast;\phi):= \int_{\Sp^1}\mathcal{L}[f_0](x,v;\phi) g(x,v; \delta_{x_\ast,v_\ast})\rd{v}\,,
\end{equation}
as a function of $x,x_\ast,v_\ast$ implicitly depend on $\phi$. Then the equation rewrites:
\begin{equation}\label{eqn:LS}
 \int_{\Omega} \beta(x,x_\ast,v_\ast;\phi)\tilde{\sigma}(x)\rd{x}=b(x_\ast,v_\ast;\phi)\,.
\end{equation}
This formulation shows that to recover $\tilde{\sigma}$ amounts to invert the first type Fredholm integral. Note that this equation holds true for every $(x_\ast,v_\ast)\in\Gamma_+$.

The equal sign rarely holds true in reality due to the data pollution. Numerically each experiment prepares one specific incoming and outgoing pair $(\phi^{(k)},\psi^{(k)})$, which uniquely defines $b^{(k)}$ and $\beta^{(k)}$ according to~\eqref{eqn:b} and~\eqref{eqn:beta}. We then seek for $\sigma$ that minimizes the following cost:
\begin{equation}\label{eqn:opt_linear}
\min_{\sigma} \frac{1}{N}\sum_k \calJ_{(k)} = \frac{1}{N}\sum_k\left(\frac{1}{2}\|\int_\Omega \beta^{(k)}\sigma\rd{x}-b^{(k)}\|^2_++\frac{\alpha}{2}\|\sigma\|_2^2\right)\,.
\end{equation} 
where we abuse the notation $\sigma$ to denote $\tilde{\sigma}$. The first term in $\calJ$ is the mismatch in~\eqref{eqn:LS} and the second term is the regularizer with a hyper-parameter $\alpha$. Both terms are measured in $L_2$.
%
%\begin{equation}\label{eqn:opt_linear}
%\min_{\sigma} \sum_k \calJ^{(k)} = \sum_k\left(\frac{1}{2}\left(\int \beta_k\sigma\rd{x}-b_k\right)^2+\frac{\alpha}{2}\|\sigma\|^2\right)\,.
%\end{equation} 
In a compact form, it writes as:
\begin{equation*}
\min_\sigma \frac{1}{N}\sum_k\left(\frac{1}{2}\|\opA_0(\sigma)[\phi^{(k)}] - \psi^{(k)}\|^2_+ +\frac{\alpha}{2}\|\sigma\|^2_2\right)\,,
\end{equation*}
where $\mathcal{A}_0$ is the linearized albedo operator that maps the incoming flow $\phi$ supported on $\Gamma_-$ to an outgoing flow measured at $(x_\ast,v_\ast)\in\Gamma_+$.
\begin{equation*}
\opA_0(\sigma)[\phi] = \int_\Omega \beta(x,x_\ast,v_\ast;\phi)\sigma(x)\rd{x}\,.
\end{equation*}

On this formulation, the application of SGD is straightforward:
\begin{equation}\label{eqn:update_linear}
\sigma_{n+1}(x) = \sigma_n(x) - \eta_n\left( \int_{\Gamma_+} \beta^{(\gamma_n)}(x,x_\ast,v_\ast)\left( \int_\Omega\beta^{(\gamma_n)}(\tilde{x},x_\ast,v_\ast)\sigma_n(\tilde{x})\rd{\tilde{x}}-b^{(\gamma_n)}(x_\ast,v_\ast)\right) \rd{x_\ast}\rd{v_\ast}+\alpha\sigma_n(x)\right)
\end{equation}
with $\gamma_n$ randomly selected from $\{1\,,\cdots,N\}$ at every step. We summarize the algorithm:

\RestyleAlgo{boxruled}
\begin{algorithm}
\SetAlgoLined
\KwData{$N$ experiments with \begin{itemize}\item[1.] incoming data $\phi^{(k)}$ for $\{k=1\,,\cdots,N\}$;\item[2.] outgoing measurements $\psi^{(k)}$ for $\{k=1\,,\cdots,N\}$;\item[3.] error tolerance $\varepsilon$;\item[4.] initial guess $\sigma_0$.\end{itemize}}
\KwResult{The minimizer $\sigma$ to the optimization problem~\eqref{eqn:opt_linear} that is within $\varepsilon$ accuracy.}
Step I: compute the dual problem~\eqref{eqn:dual_linear} using $\delta_{x_\ast,v_\ast}$ for all $(x_\ast,v_\ast)\in\Gamma^\text{d}_+$;\\
\While{$\|\frac{\rd}{\rd\sigma}\calJ_{\gamma_{n}}(\sigma_n)\|>\varepsilon$}{
Step II: randomly pick $\gamma_n\in\{1\,,\cdots,N\}$;\\
Step III: compute the background problem~\eqref{eqn:f0} using $\phi^{(\gamma_n)}$ for $f_0^{(\gamma_n)}$;\\
Step IV: compute $\beta^{(\gamma_n)}$ by~\eqref{eqn:beta};\\
Step V: compute $b^{(\gamma_n)}$ using~\eqref{eqn:b} with $\psi^{(\gamma_n)}$ and $f_0^{(\gamma_n)}$;\\
Step VI: update using~\eqref{eqn:update_linear}.\\
$n=n+1$.
}
\caption{SGD applied on the minimization problem~\eqref{eqn:opt_linear}.}\label{alg:linear}
\end{algorithm}

\subsubsection{Discretization}
We briefly describe the discrete version of~\eqref{eqn:opt_linear}. This is to replace the integration by its numerical version, and $\sigma$ and $b^{(k)}$ are replaced by their discrete counterparts as well. To be precise,
\begin{equation*}
\int_\Omega \beta^{(k)}\sigma\rd{x}\quad\rightarrow\quad \Amat^{(k)}\sigma\,,
\end{equation*}
where $\Amat^{(k)}$ is a matrix of size $n_+\times n_x$, where $n_+$ is the number of coordinates in $\Gamma_+$ and $n_x$ is the number of coordinates in $\Omega$. Its entries are defined by:
\begin{equation*}
\Amat^{(k)}_{mn} = \beta^{(k)}(x_n,x_{\ast,m},v_{\ast,m})\Delta x_n\,,
\end{equation*}
with $\Delta x_n$ being the volume grid point $x_n$ represents. For evenly distributed grids in $2D$, $\Delta x_n = \Delta x^2$ where $\Delta x$ is the mesh size. In this way:
\begin{equation*}
\left(\Amat^{(k)}\sigma\right)_m = \sum_n\beta^{(k)}(x_n,x_{\ast,m},v_{\ast,m})\sigma(x_n)\Delta x_n\,,
\end{equation*}
numerically approximates $\int \beta^{(k)}\sigma\rd{x}$ evaluated at $(x_{\ast,m},v_{\ast,m})$, the $m$-th pair on $\Gamma_+$. Notations $\sigma$ and $b^{(k)}$ are abused to denote both continuous and discrete versions.

Now the objective function becomes:
\begin{equation}
\calJ_{k}(\sigma) =\frac{1}{2} \| \Amat^{(k)}\sigma -b^{(k)} \|_2^2 + \frac{\alpha}{2} \| \sigma\|_2^2\,.
\end{equation}
Typically when rewritten in this way, $\alpha$ needs to be adjusted to incorporate the constant in the numerical integration, but we abuse the notation and still use $\alpha$.

Numerically to update in each step, one needs to take gradient of $\calJ_{k}$ with respect to $\sigma$. Given the simple form we are studying here, it is simply, denoted by $G^{(k)}$:
\begin{equation*}
G^{(k)}(\sigma) = \nabla_\sigma\calJ_{k} = \Amat^{(k)\top}\Amat^{(k)}\sigma - \Amat^{(k)\top}b^{(k)} + \alpha\sigma\,.
\end{equation*}
Denote
\begin{equation}\label{eqn:mean_k}
\mu_k :=  {\Amat^{(k)}}^\top \Amat^{(k)}\,,\quad \text{and}\quad \nu_k :=-\Amat^{(k)\top}b^{(k)}\,,
\end{equation}
then it has a simpler form:
\begin{equation*}
G^{(k)}(\sigma) = (\mu_k+\alpha)\sigma + \nu_k\,.
\end{equation*}
Note $\alpha$ is a number and $\mu_k$ is a matrix of size $n_x\times n_x$ and $\nu_k$ is a vector of $n_x$ length. We also immediately have:
\begin{equation}\label{eqn:mean_derivative}
G(\sigma) = \nabla_\sigma\calJ = \frac{1}{N}\sum_{k=1}^N\nabla_\sigma\calJ_{k} = \frac{1}{N}\sum_{k=1}^NG^{(k)}\,.
\end{equation}
Define
\begin{equation}\label{eqn:mean_A}
\mu_\Amat := \mathbb{E}[{\Amat^{(\gamma_k)}}^\top \Amat^{(\gamma_k)}] =\frac{1}{N}\sum_{k=1}^N {\Amat^{(k)}}^\top \Amat^{(k)}\,,\quad \text{and}\quad \nu_\Amat :=-\frac{1}{N}\sum_{k=1}^N\Amat^{(k)\top}b^{(k)}\,,
\end{equation}
then~\eqref{eqn:mean_derivative} has a simpler form:
\begin{equation}\label{eqn:mean_derivative2}
G(\sigma)  = \left(\mu_\Amat + \alpha\right)\sigma + \nu_\Amat\,.%\frac{1}{N}\sum_{k=1}^NA^{(k)\top}b^{(k)}\,.
\end{equation}

To update from $n$ to $n+1$ step, one randomly pick $\gamma_n$ and update $\sigma_n$ using the gradient information of $\nabla_\sigma\calJ_{\gamma_n}$:
\begin{align}\label{eqn:linear_sigma_update}
\sigma_{n+1} &= \sigma_{n} - \eta G^{(\gamma_n)}(\sigma_n)= \sigma_{n} - \eta \left( (\mu_{\gamma_n}+\alpha)\sigma_n+\nu_{\gamma_n}\right)\,.
\end{align}

\section{Error Analysis}
In this section we analyze the convergence of SGD on the linearized problem~\eqref{eqn:opt_linear}. Recall the minimization:%~\ql{please correct $N$ here}
\begin{equation}\label{eqn:linear_opt_conv}
\min_{\sigma} \calJ = \min_{\sigma} \frac{1}{N}\sum_{k=1}^N \calJ^{(k)} = \frac{1}{N}\sum_k\left(\frac{1}{2}\|\int_\Omega \beta^{(k)}\sigma\rd{x}-b^{(k)}\|^2_++\frac{\alpha}{2}\|\sigma\|_2^2\right)\,,
\end{equation}
where $\int_\Omega \beta^{(k)}\sigma\rd{x}$, upon integrating over $x\in\Omega$ provides a function supported on $(x_\ast,v_\ast)\in\Gamma_+$, and the update formula~\eqref{eqn:linear_sigma_update}.
Denote $\sigma^\ast$ the true solution to the minimization problem, meaning $G(\sigma^\ast) = 0$, and subtract it from the equation~\eqref{eqn:linear_sigma_update}, we get the updating formula for the error. Denote $e_n = \sigma_n - \sigma^\ast$, the error at $n$-th step, then:
\begin{eqnarray}
e_{n+1} &=& e_n -  \eta G^{(\gamma_n)}(\sigma_n)\nonumber\\
&=& e_n - \eta \left(G(\sigma_n)  -  G(\sigma^\ast)\right)+ \eta \left(G(\sigma_n) - G^{(\gamma_n)}(\sigma_n)\right)\nonumber\\
& =& {e_n -\eta\left(\mu_\Amat+\alpha\right)e_n}+ \eta \left(G(\sigma_n) - G^{(\gamma_n)}(\sigma_n)\right)\label{eqn:error_iteration1}\\
& =& \underbrace{e_n -\eta\left(\mu_\Amat+\alpha\right)e_n}_{\text{decay}}+ \underbrace{\eta \left( (\mu_\Amat-\mu_{\gamma_n})\sigma_n + \nu_\Amat - \nu_{\gamma_n}\right)}_{\text{fluctuation}}\label{eqn:error_iteration2}\,.
\end{eqnarray}
From the first to the second line, we used the fact that $G(\sigma^\ast) = 0$, and from the second to the third line, we use the fact that $G$ is linear on $\sigma$ as seen in~\eqref{eqn:mean_derivative2}, and definitions in~\eqref{eqn:mean_k} and~\eqref{eqn:mean_A}.

We further denote
\begin{equation}\label{eqn:def_d}
B =  \mathbb{I}-\eta \mu_\Amat -\eta \alpha\,,\quad\text{and}\quad d_n = \eta\left[(\mu_\Amat-\mu_{\gamma_n})\sigma_n + \nu_\Amat - \nu_{\gamma_n}\right]\,,
\end{equation}
then the update formula becomes:
\begin{equation}\label{eqn:update2}
e_{n+1} = Be_n + d_n\,.
\end{equation}

According to this formula, we immediately see that the decay of $e_n$ is controlled by two pieces: the first term provides the iterative decay while the second term gives fluctuation that represents the randomness from sampling $\gamma_n$. The key of error analysis is to:
\begin{itemize}
\item[1.] find appropriate $\eta$ so that $B = \mathbb{I} - \eta(\mu_\Amat+\alpha)$ has smaller than $1$ spectrum, leading to convergence;
\item[2.] show the fluctuation term has mean zero, and thus it is not producing extra error on average;
\item[3.] show the fluctuation term has very small variance, and thus the chance of producing extra error is small.
\end{itemize}
The first argument is relatively straightforward, and the latter two amount to analyze the behavior of $d_n$. We first summarize it in Lemma~\ref{lemma:d} and collect error analysis on the mean and the variance in Theorem~\ref{thm:error_mean} and Theorem~\ref{thm:error_variance} respectively.

%
%
%In this section we present an error analysis on the linearized problem \eqref{eqn:opt_linear}  in the discretized setting and derive a necessary condition on learning rate $\eta$ for convergence of SGD algorithm. We discretize the domain $\Omega$ into grid points $\{x_1,\ldots,x_{N_x}\}$ and for simplicity we denote by $\sigma \in \mathbb{R}^{N_x}$ its discretized version. Considering the linearized problem \eqref{alg:linear} in discretized setting, the minimization problems becomes:
%\begin{equation}\label{eqn:discrete_opt}
%\min_{\sigma} J(\sigma) = \min_{\sigma} \frac{1}{N}\sum_{k=1}^N J_k(\sigma)
%\end{equation}
%where we define that
%\begin{equation}
%\begin{aligned}
%J_k(\sigma) &= \frac{1}{2} \| \int_\Omega \beta^{(k)}\sigma \rd{x}-\psi^{(k)} \|_{+}^2 +\frac{1}{2}\alpha \|\sigma \|^2 \\
%:&=\frac{1}{2} \| A^{(k)}\sigma -b^{(k)} \|_2^2 + \frac{1}{2}\tilde{\alpha} \| \sigma\|_2^2  
%\end{aligned}
%\end{equation}
%Here $A^{(k)}$ depends on $\beta^{(k)}$ and the numerical integration scheme on $\Omega$ and $\Gamma_+$, $\tilde{\alpha}$ depends on $\alpha$ and numerical integration scheme on $\Omega$ and $\| \cdot \|_2$ is the vector $2$-norm. 

\begin{lemma}\label{lemma:d}
Assume
\begin{equation*}
\mathbb{E}\left(\|\mu_\Amat-\mu_{\gamma_n}\|^2_2\right) < C_\mu\,,\text{and}\quad\mathbb{E}\left(\|\nu_\Amat-\nu_{\gamma_n}\|^2_2\right)<C_\nu\,,\quad\forall n\,.
\end{equation*}
Using the definition in~\eqref{eqn:def_d} we have:
\begin{itemize}
\item[1.] $\mathbb{E}(d_n) = 0$ for all $n$;
\item[2.] $\text{Cov}[d_i,d_j] = 0$ for all $1\leq i< j\leq n$;
\item[3.] $\text{Cov}[d_n,d_n] \leq C \eta^2 (\mathbb{E}\left(\| \sigma_n \|_2^2\right) + 1)$\,.
\end{itemize}
\end{lemma}
\begin{proof}
\begin{itemize}
\item[1.]
According to the definition:
\begin{equation*}
\frac{1}{\eta}\mathbb{E}(d_n) = \mathbb{E}((\mu_\Amat-\mu_{\gamma_n})\sigma_n) + \mathbb{E}(\nu_\Amat-\nu_{\gamma_n})\,.
\end{equation*}
The second term is zero due to equation~\eqref{eqn:mean_k} and~\eqref{eqn:mean_A}. To study the first term we first realize that the randomness comes from both $\gamma_n$ and $\sigma_n$. Due to~\eqref{eqn:update}, $\sigma_n$ only depends on $\{\gamma_1,\ldots,\gamma_{n-1}\}$, and thus it is independent of $\gamma_n$. Therefore:
\begin{equation*}
\mathbb{E}((\mu_\Amat-\mu_{\gamma_n})\sigma_n)=\mathbb{E}(\mu_\Amat-\mu_{\gamma_n})\mathbb{E}(\sigma_n)\,.
\end{equation*}
Given~\eqref{eqn:mean_k} and~\eqref{eqn:mean_A}, we see $\mathbb{E}((\mu_\Amat-\mu_{\gamma_n})\sigma_n)=0$ and thus $d_n$ is mean zero.
\item[2.] Since $d_i$ is mean zero:
\begin{align*}
\Cov[d_i,d_j] = \mathbb{E}(d_i d_j^\top) = \mathbb{E}\left(d_i\mathbb{E}(d_j^\top)\right)=0\,.
\end{align*}
The first equation comes from $d_i$ and $d_j$ being mean zero. The second equation holds true because $i<j$.
\item[3.] For the third covariance:
\begin{align*}
\frac{1}{\eta^2}\Cov[d_n,d_n] =& \frac{1}{\eta^2}\mathbb{E}(d_n d_n^\top) \\
=& \mathbb{E}\left((\mu_\Amat-\mu_{\gamma_n})\sigma_n\sigma_n^\top(\mu_\Amat-\mu_{\gamma_n})^\top\right) + \mathbb{E}\left((\nu_\Amat-\nu_{\gamma_n})(\nu_\Amat-\nu_{\gamma_n})^\top\right)\\
&+ \mathbb{E}\left((\mu_\Amat-\mu_{\gamma_n})\sigma_n(\nu_\Amat-\nu_{\gamma_n})^\top\right)+\mathbb{E}\left((\nu_\Amat-\nu_{\gamma_n}) \sigma_n^\top (\mu_\Amat-\mu_{\gamma_n})^\top\right) \,.
\end{align*}
Take arbitrary $x\in \mathbb{R}^{N_x}$ with $\| x \|_2 = 1$ and multiply on both sides, we have
\begin{align*}
\frac{1}{\eta^2}x^\top\Cov[d_n,d_n]x = &  x^\top \mathbb{E}\left((\mu_\Amat-\mu_{\gamma_n})\sigma_n\sigma_n^\top(\mu_\Amat-\mu_{\gamma_n})^\top\right) x + x^\top \mathbb{E}\left((\nu_\Amat-\nu_{\gamma_n})(\nu_\Amat-\nu_{\gamma_n})^\top\right) x \\
&+2 x^\top \mathbb{E}\left((\mu_\Amat-\mu_{\gamma_n})\sigma_n(\nu_\Amat-\nu_{\gamma_n})^\top\right) x\\
\leq & 2C_\mu \mathbb{E} \left( \| \sigma_n \|_2^2 \right) + 2C_\nu \,.
\end{align*}
To obtain the inequality we used the fact that
\begin{align*}
x^\top \mathbb{E}\left((\mu_\Amat-\mu_{\gamma_n})\sigma_n\sigma_n^\top(\mu_\Amat-\mu_{\gamma_n})^\top\right) x  =& x^\top \mathbb{E}\left((\mu_\Amat-\mu_{\gamma_n})\mathbb{E}\left(\sigma_n\sigma_n^\top \right) (\mu_\Amat-\mu_{\gamma_n})^\top\right) x   \\
\leq & C_\mu \mathbb{E}\left(\|\sigma_n \|_2^2\right)
\end{align*}
and that
\begin{align*}
2 x^\top \mathbb{E}\left((\mu_\Amat-\mu_{\gamma_n})\sigma_n(\nu_\Amat-\nu_{\gamma_n})^\top\right) x =& 2 \mathbb{E}\left(x^\top(\mu_\Amat-\mu_{\gamma_n})\mathbb{E}(\sigma_n)(\nu_\Amat-\nu_{\gamma_n})^\top x\right)\\
\leq & \mathbb{E}\left[\left(x^\top(\mu_\Amat-\mu_{\gamma_n})\mathbb{E}(\sigma_n)\right)^2 + \left((\nu_\Amat-\nu_{\gamma_n})^\top x\right)^2\right]\\
\leq & C_\mu \mathbb{E} \left( \| \sigma_n \|_2^2 \right) + C_\nu \,.
\end{align*}
%\begin{align*}
%2 x^\top \mathbb{E}\left((\mu_\Amat-\mu_{\gamma_n})\sigma_n(\nu_\Amat-\nu_{\gamma_n})^\top\right) x \leq &2\mathbb{E} \left( \|\mu_\Amat -\mu_{\gamma_n} \|_2 \|\nu_\Amat -\nu_{\gamma_n} \|_2 \right)\left(\mathbb{E}\left( \| \sigma_n \|_2^2 \right)\right)^{1/2} \\
%\leq & C_\mu \mathbb{E} \left( \| \sigma_n \|_2^2 \right) + C_\nu 
%\end{align*}
We achieve the conclusion by multiplying $\eta^2$ on both sides and choose $C = 2\max\{C_\mu\,,C_\nu\}$.
\end{itemize}
\end{proof}

With this lemma we study the mean and the variance of the error in the following two theorems.
\begin{theorem}\label{thm:error_mean}
Denote $\sigma^\ast$ the minimizer of problem \eqref{eqn:linear_opt_conv} and the expected value of error:
\begin{equation*}
u_n = \mathbb{E}(e_n) = \mathbb{E}(\sigma_n-\sigma^\ast)\,.
\end{equation*}
Assume that $\mu_\Amat$ (defined in~\eqref{eqn:mean_A} has a bounded spectrum, meaning there exists $C_A$ such that:
\begin{equation}
\|\mu_\Amat\|_2\leq C_A\,,
\end{equation}
then for $0<\eta<\frac{2}{C_A+\alpha}$, the expected value of error decays to zero exponentially fast:
\begin{equation}\label{eqn:u_n}
\| u_n \|_2 \leq \lambda^n \| u_0 \|_2 \,,
\end{equation}
where $|\lambda | <1$ will be defined in \eqref{eqn:lambda}.
\end{theorem}
\begin{proof}
We start from the iteration formula for $e_n$ in~\eqref{eqn:update}. Take expectation on both sides:
\begin{equation}\label{eqn:SGD_iter3}
u_{n+1} = u_n - \eta(\mu_\Amat+\alpha)u_n +\eta \mathbb{E}(d_n)\,.
\end{equation}
Since $d_n$ is mean zero according to the previous lemma,~\eqref{eqn:SGD_iter3} becomes:
\begin{equation*}
u_{n+1} = (\mathbb{I} - \eta\mu_\Amat - \eta\alpha)u_n\,.
\end{equation*}
With $0<\eta<\frac{2}{C_A+\alpha}$ and define
\begin{equation}\label{eqn:lambda}
\lambda := \|\mathbb{I} - \eta \mu_\Amat- \eta {\alpha}\|_2\,,
\end{equation}
$\lambda$ is guaranteed to be controlled by $1$ and we achieve the conclusion.
\end{proof}

\begin{theorem}\label{thm:error_variance}
With small learning rate $\eta$, the error of SGD algorithm has bounded covariance:
\[
\Cov[e_n,e_n] \lesssim \eta\,,\quad\forall n\,.
\]
\end{theorem}
%Denote the error $e_n = \sigma_n -\sigma^\ast$ and assume that
%\[
%\mathbb{E} [\| {A^{(\gamma_k)}}^\top A^{(\gamma_k)}\|_2 ] \leq C_A \quad \text{and} \quad {\Cov}[{A^{(\gamma_k)}}^\top b^{(\gamma_k)},{A^{(\gamma_k)}}^\top b^{(\gamma_k)}] \leq C_b \mathbb{I}
%\]
%then for small learning rate $\eta$,
\begin{proof}
We once again use:
\begin{equation*}
e_{n+1} = Be_n + d_n\,,
\end{equation*}
with
\begin{equation*}
B =  \mathbb{I}-\eta \mu_\Amat -\eta {\alpha}\,,\quad\text{and}\quad d_n = (\mu_\Amat-\mu_{\gamma_n})\sigma_n + \nu_\Amat - \nu_{\gamma_n}\,.
\end{equation*}
By induction,
\begin{equation*}
e_n = B^n e_0 + \sum_{j=1}^{n-1} B^{n-j} d_{j}  \,.
\end{equation*}
Take covariance of both sides and recall $\Cov[d_i,d_j] = 0$ for all $i\neq j$:
\begin{equation*}
\text{Cov}[e_n,e_n] = \sum_{i,j} \text{Cov}[B^{n-i} d_i\,,B^{n-j} d_{j}] = \sum_i B^{n-i} \Cov[d_{i}\,,d_{i}] (B^{n-i})^\top   \,.
\end{equation*}
Take arbitrary $x\in \mathbb{R}^{N \times 1}$ with $\|x\|_2 = 1$ and multiply on both sides, we have
\begin{equation*}
x^\top \text{Cov}[e_n,e_n] x = \sum_i (x^\top B^{n-i}) \text{Cov}[d_{i}\,,d_{i}] (x^\top B^{n-i})^\top \leq \sum_i C\eta^2 \lambda^{2(n-i)} (\mathbb{E}[\|\sigma_{i}\|_2^2]+1) \,.
\end{equation*}
where the inequality incorporates the previous lemma. Further notice that $\mathbb{E}[\| \sigma_{i} \|_2^2] \leq \mathbb{E}[\| e_{i} \|_2^2] + \| \sigma^\ast \|_2^2$, we absorb the constant:
\begin{equation}\label{eqn:variance_est}
x^\top \text{Cov}[e_n,e_n] x \leq \sum_i \tilde{C}\eta^2 \lambda^{2(n-i)} (\mathbb{E}[\|e_{i}\|_2^2]+1)\,,
\end{equation}
where $\tilde{C} = C+\|\sigma^\ast\|_2$. This inequality only serves as a iterative formula. Upon assuming $\mathbb{E}[\|e_i\|^2_2]$ is uniformly bounded by $M$, then:
\begin{equation}
x^\top \text{Cov}[e_n,e_n] x \leq \tilde{C} \eta^2 (M+1) \frac{1-\lambda^{2n}}{1-\lambda^2} \lesssim \eta\,.
\end{equation}
The last inequality comes from the definition of $\lambda = \|\mathbb{I}-\eta\mu_\Amat-\eta\alpha\|_2=\mathcal{O}(1-\eta)$. Since $x$ is arbitrary, we achieve the conclusion.

To show that there exists a constant $M>0$ such that $\mathbb{E}[\|e_i\|^2_2]$ is truly uniformly bounded by $M$, we use mathematical induction. It is easy to prove the argument is true for $i=0$ by choosing $M=2\max\{1,\mathbb{E}[\|e_0\|_2^2]\} = \max\{2,2\|e_0\|_2^2\}$. Then we assume the argument is true for all $i<n$ and we want to show that $\mathbb{E}[\|e_n\|_2^2]\leq M$. We notice that
\[
\mathbb{E}[\|e_n\|^2_2] = \text{Tr}(\text{Cov}[e_n,e_n]) \leq N \| \text{Cov}[e_n,e_n] \|_2 \,,
\]
then since \eqref{eqn:variance_est} is true for any $x$, we have
\[
 \| \text{Cov}[e_n,e_n] \|_2 \leq  \sum_{i} \tilde{C}\eta^2 \lambda^{2(n-i)} (\mathbb{E}[\|e_{i}\|_2^2]+1) \,.
\]
Combine the above two inequalities and our induction assumption for $i<n$, we derive that
\[
\mathbb{E}[\|e_n\|^2_2] \leq N\left(\tilde{C} \eta^2 (M+1) \frac{1-\lambda^{2(n-1)}}{1-\lambda^2} +\tilde{C}\eta^2 (\mathbb{E}[\| e_n \|_2^2]+1)\right) \,.
\]
For small enough $\eta$, this leads to:
\[
\mathbb{E}[\|e_n\|^2_2] \leq 2N\left(\tilde{C} \eta^2 (M+1) \frac{1-\lambda^{2(n-1)}}{1-\lambda^2} +\tilde{C}\eta^2 \right)\,.
\]
Use the fact $\lambda = \mathcal{O}(1-\eta)$, we can further choose $\eta$ small such that
\[
\mathbb{E}[\|e_n\|^2_2] \leq \frac{1}{2}(M+1) +\frac{1}{2} = \frac{M}{2} + 1 \leq M\,,
\]
which finishes the mathematical induction.

\end{proof}

We finally comment that the two theorems above in fact resonate the analysis in the general setting as stated in Section~\ref{sec:SGD}. There are two main pieces in the error: the iterative decaying term, and the fluctuation term. If the initial guess gives an order 1 error, then the decaying term dominates first, and one simply see the error converging to zero exponentially fast. Once the error becomes as small as the variance (which is at $\eta$ level), the fluctuation term dominates. To force the error converging to zero, numerically one could gradually decrease $\eta$ so that the error fluctuates around zero with smaller and smaller variance. The result will be seen in our numerical results too.

%%%%%%%%%%%%%%%%%%%%%%%%%%%%%%%%%%%%%%%%%%%%%%%
\section{Numerical test}
To illustrate our theoretical results, we present a few numerical test below. The computational space domain is a unit square $\Omega=[0,1]^2$ with mesh size $\rd{x} = 1/20=0.05$, and the velocity domain a unit circle $\Sp$ with mesh size $\rd{\theta} = \frac{2\pi}{40}$. Therefore in the discrete setting:
\begin{equation*}
\Omega^d\times\Sp^d = \{(x_m,\theta_n) = (m_1\rd{x},m_2\rd{x},-\pi+n\rd\theta): \text{with}\quad m_1,m_2 = 0,\cdots,20\,,n = 0,\cdots,40\}\,,
\end{equation*}
and
\begin{equation*}
\Gamma^d_- =\{(x_1 = m_1\rd{x}\,,x_2 = m_2\rd{x}\,,\theta = n\rd\theta)\}
\end{equation*}
with
\begin{equation*}
\begin{aligned}
&  ~\{m_1=0,m_2\in[0,20],n\in[10,30]\}\cup \{m_1=20,m_2\in[0,20],n\in\left([0,10]\cup[30,40]\right)\} \\
				   &\cup \{m_1\in[0,20],m_2=0,n\in[20,40]\} \cup\{m_1\in[0,20],m_2=20,n\in[0,20]\}\,.
\end{aligned}
\end{equation*}

We use GMRES~\cite{LW17} to solve the forward problem \eqref{eqn:rte_2d} with tolerance $10^{-12}$. The scattering coefficient in our experiment is set to be
\begin{equation}
\sigma(x_1,x_2) = \frac{1}{20}\left[1+8\exp\left(-10(x_1-\frac{1}{4})^2-10(x_2-\frac{1}{4})^2\right)+4\exp\left(-10(x_1-\frac{3}{4})^2-10(x_2-\frac{3}{4})^2\right)\right]\,.
\end{equation}
Its evaluation in $\Omega$ ranges from 0.05 to 0.45, as plotted in Figure~\ref{fig:RealSigma}.
\begin{figure}
\centering
\includegraphics[width=0.6\textwidth]{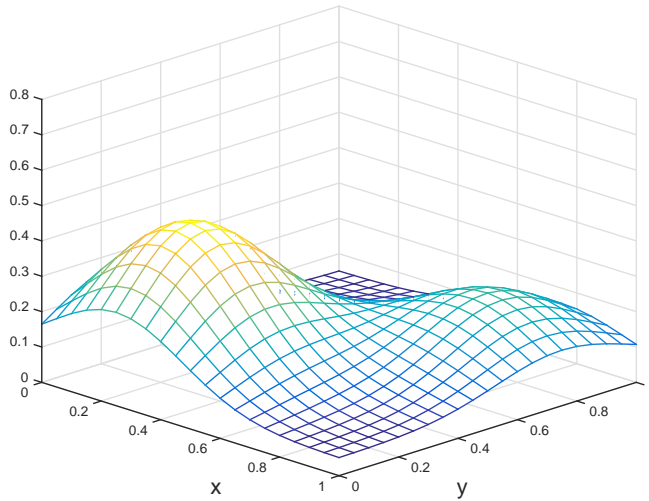}
\caption{Real Scattering Coefficient}\label{fig:RealSigma}
\end{figure}
\subsection{Nonlinear case}
In the nonlinear case \eqref{eqn:opt_nonlinear}, we use 1000 data points $\{(\phi^{(j)},\psi^{(j)}): 1\leq j\leq 1000\}$, where $\phi^{(j)}(x,v)$ is a Dirac delta function centered at a random boundary point and pointing to a random inflow direction. $\psi^{(j)}(x,v)$ is the corresponding measurement on the outflow boundary, i.e.
\begin{equation}
\phi^{(j)}(x,v) = \delta(x-x_{(j)})\delta(v-v_{(j)}), (x_{(j)},v_{(j)})\in \Gamma_- \quad \text{and}\quad \psi^{(j)}(x,v) = (n_x\cdot v) f(x,v;\phi^{(j)})|_{\Gamma_+} \,.
\end{equation}
For our numerical experiments, we set the regularization parameter $\alpha=1$ and learning rate $\eta_n = \frac{\eta_0}{1+\eta_0 \alpha n}$ with $\eta_0 = 0.0044$. Note that the learning rate is a hyperparameter that can be adjusted according to users' preferences. We choose the recommended $\frac{1}{n}$ from~\cite{Bottou2012}. We test our algorithm with two different initial guesses: 1. Initial guess is a constant deviation from the real scattering coefficient $\sigma_0 = \sigma + 0.18$; 2. Initial guess is the product of the scattering coefficient and a random field: $\sigma_0 = \sigma R$, where $R\in \mathbb{R}^{21\times 21}$ has i.i.d. random variable components drew from uniform distribution $U([0.1,3.1])$. In each iteration, two forward problems (one original and one dual) are solved to compute the gradient and we run SGD algorithm for 2000 steps.

We present the numerical solutions in Figure~\ref{fig:Const} and Figure~\ref{fig:Ran} for constant deviation and random deviation as the initial guess respectively. In both, the upper left plot shows the initial guess $\sigma_0$, and the difference compared with the true media is plotted in the upper right. The lower left and lower right plots show the numerical solution after $2000$ iterations and its difference from the true media. We also record the relative error between $\sigma_n$ and $\sigma$ and plot the decay in Figure~\ref{fig:ReEr}. Note that due to the nontrivial regularization term, we cannot expect the solution converging to the true media. As seen in Figure~\ref{fig:ReEr} the error saturates at $0.2$. It does provide very good recovery visually as seen in Figure~\ref{fig:Const} and~\ref{fig:Ran}.

\begin{figure*}
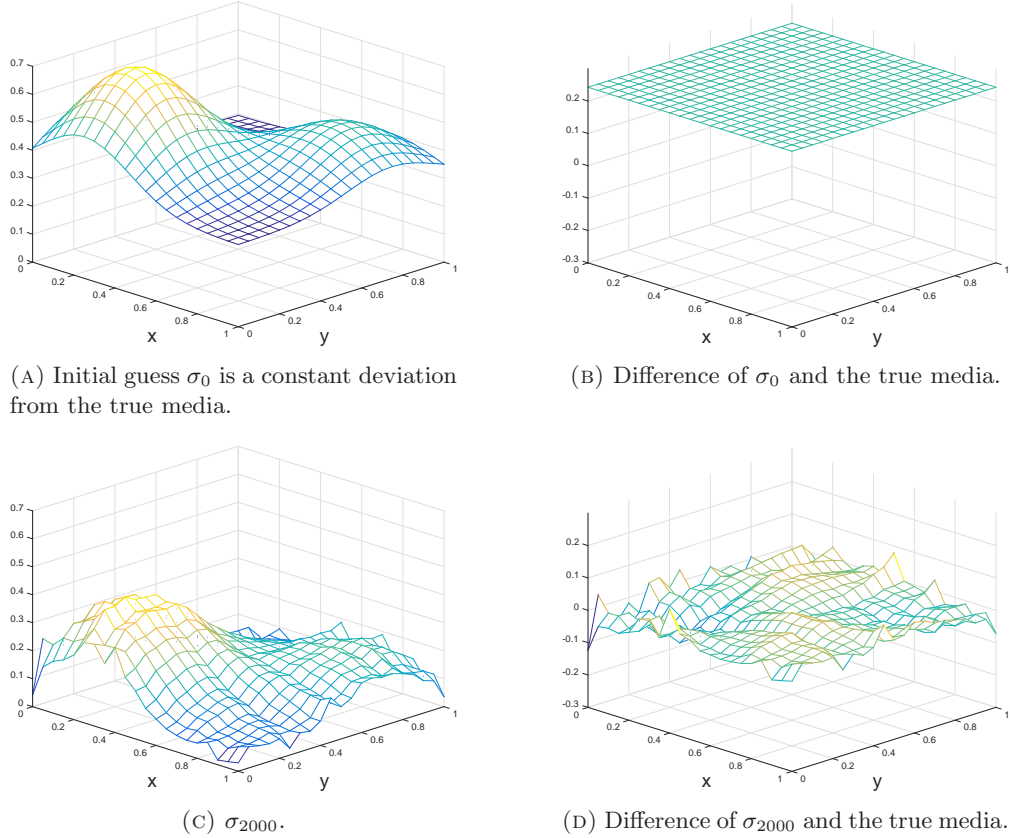

\centering

\begin{subfigure}[t]{0.35\textwidth}
\centering
\includegraphics[width=\textwidth]{Init_Const.eps}
\caption{Initial guess $\sigma_0$ is a constant deviation from the true media.}
\label{fig:Init_Const}
\end{subfigure}%
\hspace{0.5in}
\begin{subfigure}[t]{0.35\textwidth}
\centering
\includegraphics[width=\textwidth]{Dif_Init_Const.eps}
\caption{Difference of $\sigma_0$ and the true media.}
\label{fig:Dif_Init_Const}
\end{subfigure}

\bigskip 

\begin{subfigure}[t]{0.35\textwidth}
\centering
\includegraphics[width=\textwidth]{Final_Const.eps}
\caption{$\sigma_{2000}$.}
\label{fig:Final_Const}
\end{subfigure}%
\hspace{0.5in}
\begin{subfigure}[t]{0.35\textwidth}
\centering
\includegraphics[width=\textwidth]{Dif_Final_Const.eps}
\caption{Difference of $\sigma_{2000}$ and the true media.}
\label{fig:Dif_Final_Const}
\end{subfigure}

\caption{Nonlinear setting with initial guess being a constant deviation from the true media.}
\label{fig:Const}
\end{figure*}

\begin{figure*}
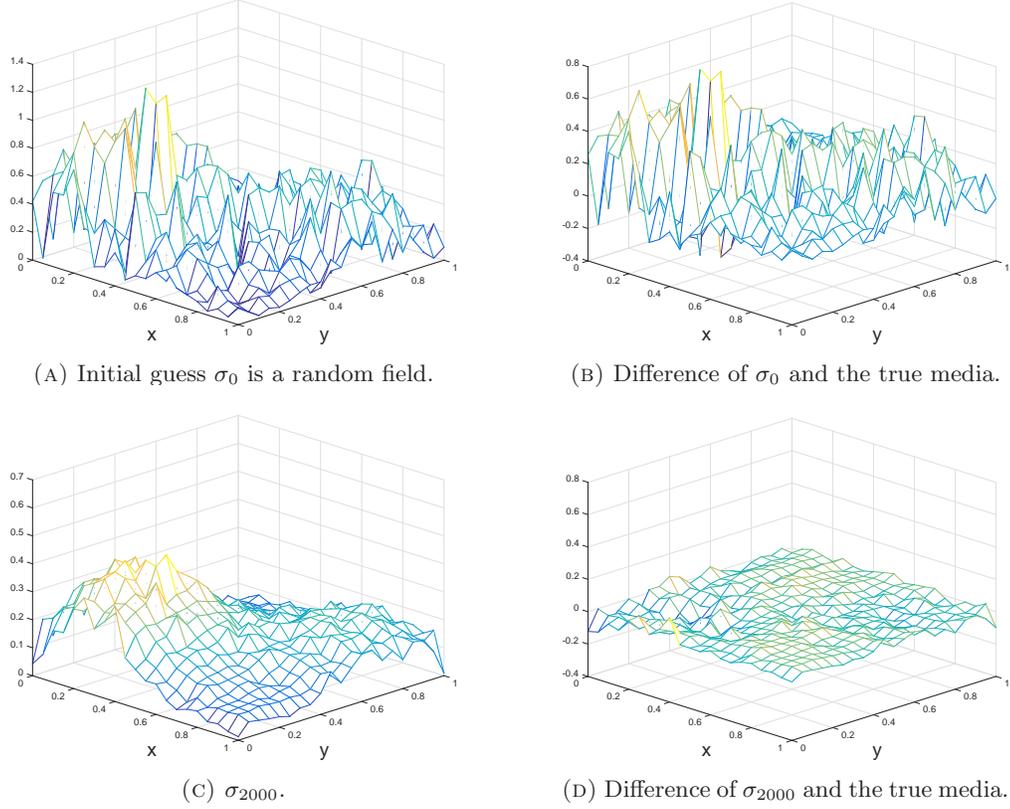

\centering

\begin{subfigure}[t]{0.35\textwidth}
\centering
\includegraphics[width=\textwidth]{Init_Ran.eps}
\caption{Initial guess $\sigma_0$ is a random field.}
\label{fig:Init_Ran}
\end{subfigure}%
\hspace{0.5in}
\begin{subfigure}[t]{0.35\textwidth}
\centering
\includegraphics[width=\textwidth]{Dif_Init_Ran.eps}
\caption{Difference of $\sigma_0$ and the true media.}
\label{fig:Dif_Init_Ran}
\end{subfigure}

\bigskip 

\begin{subfigure}[t]{0.35\textwidth}
\centering
\includegraphics[width=\textwidth]{Final_Ran.eps}
\caption{$\sigma_{2000}$.}
\label{fig:Final_Ran}
\end{subfigure}%
\hspace{0.5in}
\begin{subfigure}[t]{0.35\textwidth}
\centering
\includegraphics[width=\textwidth]{Dif_Final_Ran.eps}
\caption{Difference of $\sigma_{2000}$ and the true media.}
\label{fig:Dif_Final_Ran}
\end{subfigure}

\caption{Nonlinear setting with initial guess being a random field.}
\label{fig:Ran}
\end{figure*}

\begin{figure}
\centering
\includegraphics[width=0.5\textwidth]{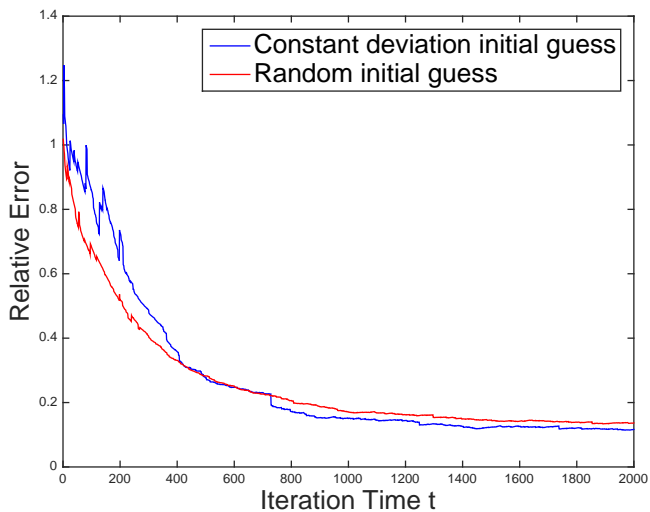}
\caption{Nonlinear setting. The convergence of relative error in time. We see that the error decays almost exponentially fast at the beginning with small fluctuations and gradually saturate. The learning rate $\eta_n$ is extremely small after $1000$ times steps and the decay significantly slows down.}
\label{fig:ReEr}
\end{figure}

\subsection{Linear Case}
We use the same data set in the linearized setting. The background state is given as proportional to the real media $\sigma_0=0.95\sigma$, and thus the to-be-recovered perturbed media $\tilde{\sigma}$, by definition~\eqref{eqn:sigma_tilde} ranges from $0.0025$ to $0.0225$. We choose same regularization coefficient $\alpha=1$. We also test the problem using the constant learning rate $\eta_0= 0.0002$ and the learning rate recommended in~\cite{Bottou2010}: $\eta_n = \frac{\eta_0}{1+\eta_0\alpha n}$ with $\eta_0 = 0.0002$.

We once again use constant deviation and random deviation as the initial guess for the SGD algorithm. For constant deviation initial guess we set $\tilde{\sigma}_0 = \tilde{\sigma}+0.0111$ whereas for random initial guess we set $\tilde{\sigma}_0 = \tilde{\sigma}R$ with $R\in \mathbb{R}^{21\times 21}$ drew its components from uniform distribution $U([1, 3])$.

As presented in Algorithm~\ref{alg:linear}, several \textit{offline} adjoint problems are pre-computed using background state $\sigma_0$ with Dirac delta outflow boundary conditions. In each iteration, only one forward problem is solved using background state $\sigma_0$ and random input $\phi^{(\gamma^n)}$ for $f_0(x,v;\phi^{(\gamma^n)})$. We run SGD algorithm with 20000 iterations. The numerical results are demonstrated in Figure~\ref{fig:Lin_Const} and Figure~\ref{fig:Lin_Ran}. They have constant and random deviation as the initial guess respectively. The decay of the relative error for both types of learning rates are shown in Figure~\ref{fig:Lin_ReEr}. In Figure~\ref{fig:Large_Deviation} we plot and compare the convergence of the error when the initial guess largely deviates from the true solution: $\sigma_0 = 0.2000$. The initial relative error is as large as $17.12$.

Comparing to the nonlinear case, the convergence of relative error requires more iterations as here we aim to recover the small residue $\tilde{\sigma}= \sigma-\sigma_0$, which is much smaller than $\sigma$.

\begin{figure*}
\centering

\begin{subfigure}[t]{0.35\textwidth}
\centering
\includegraphics[width=\textwidth]{Lin_Init_Const.eps}
\caption{Initial guess $\tilde{\sigma}_0$ is a constant deviation.}
\label{fig:Lin_Init_Const}
\end{subfigure}%
\hspace{0.5in}
\begin{subfigure}[t]{0.35\textwidth}
\centering
\includegraphics[width=\textwidth]{Lin_Dif_Init_Const.eps}
\caption{Difference of $\tilde{\sigma}_0$ with the true media $\tilde{\sigma}$.}
\label{fig:Lin_Dif_Init_Const}
\end{subfigure}

\bigskip 

\begin{subfigure}[t]{0.35\textwidth}
\centering
\includegraphics[width=\textwidth]{Lin_Final_Const.eps}
\caption{$\tilde{\sigma}_{20000}$.}
\label{fig:Lin_Final_Const}
\end{subfigure}%
\hspace{0.5in}
\begin{subfigure}[t]{0.35\textwidth}
\centering
\includegraphics[width=\textwidth]{Lin_Dif_Final_Const.eps}
\caption{Difference of $\tilde{\sigma}_{20000}$ with the true media $\tilde{\sigma}$.}
\label{fig:Lin_Dif_Final_Const}
\end{subfigure}

\caption{Linearized setting with initial guess being a constant deviation from the true media.}
\label{fig:Lin_Const}
\end{figure*}

\begin{figure*}
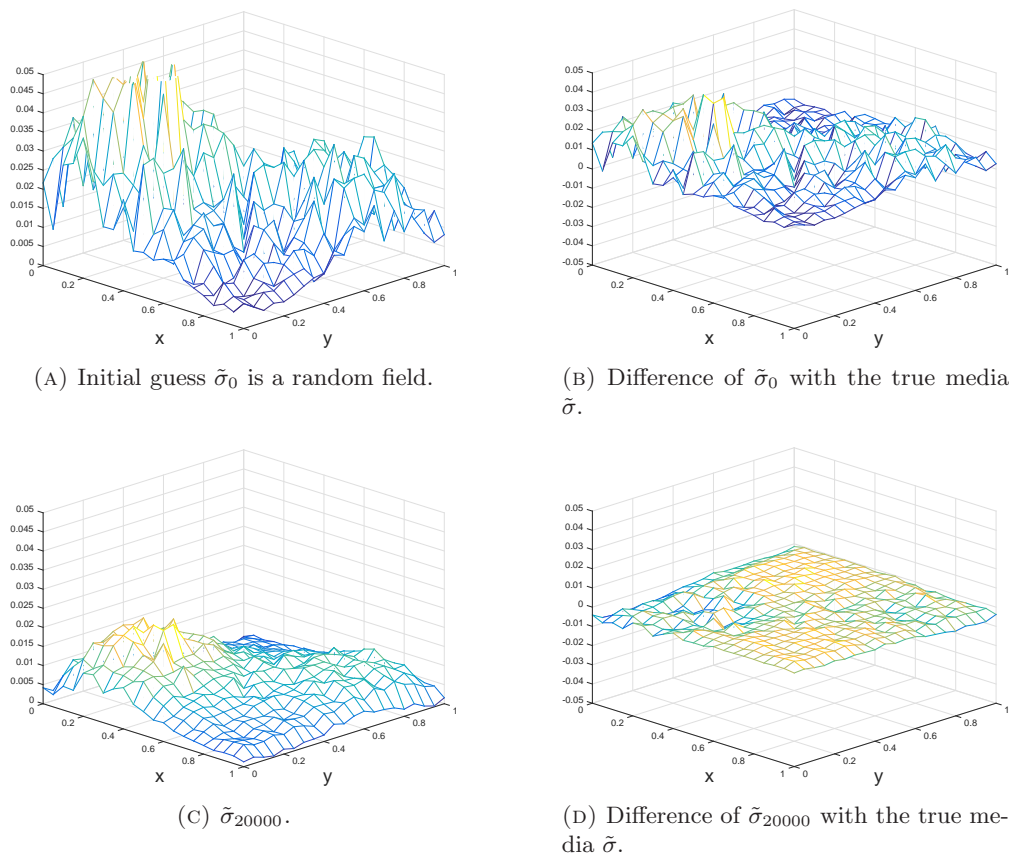

\centering

\begin{subfigure}[t]{0.35\textwidth}
\centering
\includegraphics[width=\textwidth]{Lin_Init_Ran.eps}
\caption{Initial guess $\tilde{\sigma}_0$ is a random field.}
%\label{fig:Lin_Init_Ran}
\end{subfigure}%
\hspace{0.5in}
\begin{subfigure}[t]{0.35\textwidth}
\centering
\includegraphics[width=\textwidth]{Lin_Dif_Init_Ran.eps}
\caption{Difference of $\tilde{\sigma}_0$ with the true media $\tilde{\sigma}$.}
%\label{fig:Lin_Dif_Init_Ran}
\end{subfigure}

\bigskip 

\begin{subfigure}[t]{0.35\textwidth}
\centering
\includegraphics[width=\textwidth]{Lin_Final_Ran.eps}
\caption{$\tilde{\sigma}_{20000}$.}
%\label{fig:Lin_Final_Ran}
\end{subfigure}%
\hspace{0.5in}
\begin{subfigure}[t]{0.35\textwidth}
\centering
\includegraphics[width=\textwidth]{Lin_Dif_Final_Ran.eps}
\caption{Difference of $\tilde{\sigma}_{20000}$ with the true media $\tilde{\sigma}$.}
%\label{fig:Lin_Dif_Final_Ran}
\end{subfigure}

\caption{Linearized setting with initial guess being a random field.}
\label{fig:Lin_Ran}
\end{figure*}

\begin{figure}
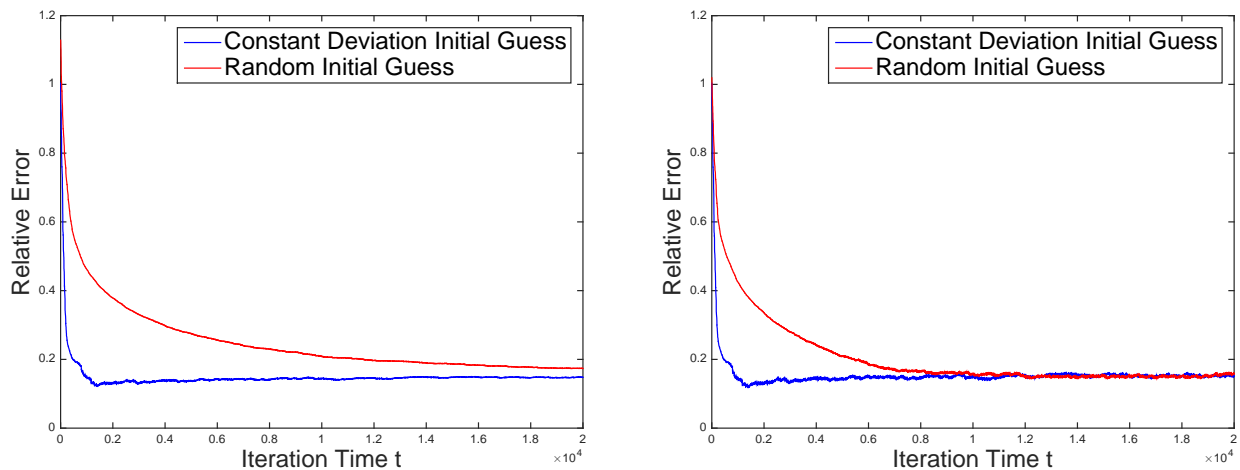

\centering
\begin{subfigure}{0.5\textwidth}
	\includegraphics[width=0.9\textwidth]{Lin_RelativeError.eps}
\end{subfigure}%
~
\begin{subfigure}{0.5\textwidth}
	\includegraphics[width=0.9\textwidth]{ConLR_Lin_RelativeError.eps}
\end{subfigure}
\caption{Linearized setting. The convergence of relative error in time. The error decays almost exponentially fast at the beginning with small fluctuations and gradually saturate. The two panels are for changing-in-time learning rate and the constant learning rate respectively.}
\label{fig:Lin_ReEr}
\end{figure}

\begin{figure}
\centering
\begin{subfigure}{0.5\textwidth}
	\includegraphics[width=0.9\textwidth]{Large_Deviation.eps}
\end{subfigure}%
~
\begin{subfigure}{0.5\textwidth}
	\includegraphics[width=0.9\textwidth]{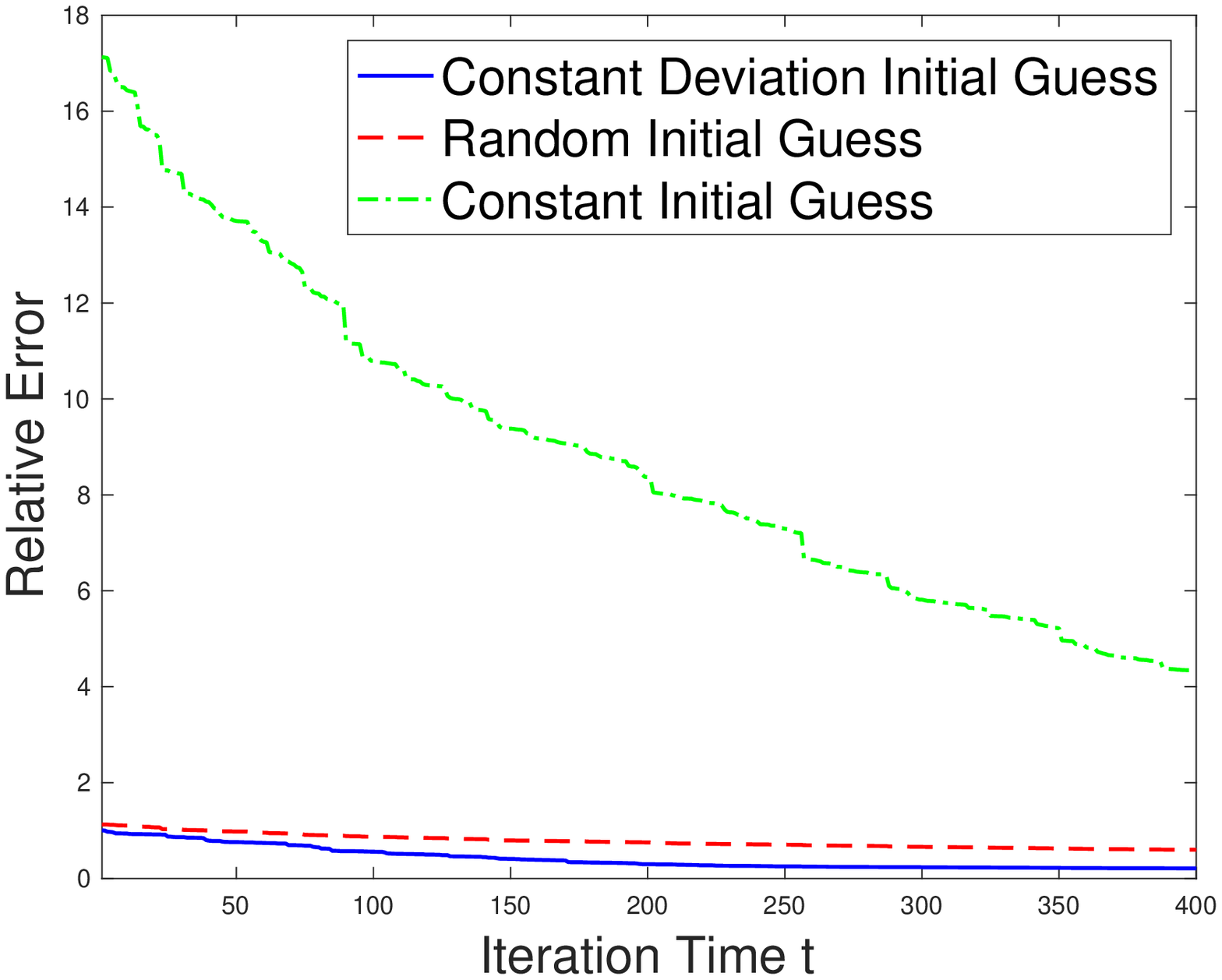}
\end{subfigure}
\caption{Linearized setting. The green dashed line shows the convergence of relative error with the initial guess far away from the true solution. The plot on the right is the zoom-in to the first $400$ steps. The oscillation introduced from the stochasticity in the algorithm is obvious.}
\label{fig:Large_Deviation}
\end{figure}

\subsection{Numerical cost study}
We dedicate this subsection for comparing numerical cost of SGD and the classical GD method. Initial guess is set as $\sigma_0 = \sigma R$ with $R\in\mathbb{R}^{21\times 21}$ drew from uniform distribution $U([0.1,3.1])$. Regularizer $\alpha = 1$ and learning rate $\eta_0 = 0.0044$. Both SGD and GD are used for the optimizer with the sample size $N$ being $100$, $200$ and $400$. The computation is terminated once error tolerance $\text{TOL}=0.2$ is reached, or maximum number of iteration is achieved. We set maximum number of iteration $2000$ for SGD and $100$ for GD.

In Table ~\ref{tab:cost_SGD} we record the number of RTEs that need to compute per iteration, the number of iterations needs to achieve convergence, and the total number RTEs computed for all three sample sizes, and both methods. Note that in each iteration, SGD requires computation of one forward RTE~\eqref{eqn:forward} and one dual RTE~\eqref{eqn:dual}, while GD requires computation of $N$ forward and $N$ duals. Note also that with $N=100$ both SGD and GD fail to converge before achieving the maximum number of allowed iterations.

%\blue{
%	We also study the numerical cost of the proposed nonlinear stochastic solver and its deterministic GD version with respect to $N$, the number of experiments. For the stochastic solver, we choose $\alpha = 1$, $\gamma_0 = 0.0044$ and $\sigma_0=\sigma R$ where $R\in\mathbb{R}^{21\times 21}$ has i.i.d. random variable components drew from uniform distribution $U([0.1,3.1])$. Let $N$ ranges from $100$, $200$ and $400$. In Table~\ref{tab:cost_SGD} we record the number of iteration needed for the convergence to $\text{TOL} = 0.2$ accuracy tolerance with maximum number of iterations $T=2000$. Notice that for $N=100$, the SGD algorithm does not reach error tolerance and stop at $T=2000$ with relative error $0.2037$.
%}
%
%\blue{	
%	For the deterministic GD solver, we choose $\alpha = 1$, $\gamma_0 = 0.0441$, same random initial data $\sigma_0=\sigma R$ and $N = 100,200,400$. In Table~\ref{tab:cost_GD} we record number of RTE solver (as well as dual problem) used per iteration and number of iteration for the convergenc to $\text{TOL}=0.2$ and $T=100$. In factt for $N=100$, the GD algorithm does not reach error tolerance and eventually diverges.
%}
%
%\blue{
%	Comparing these two algorithms, we see that with $N=200$, $87$ iterations are needed with $2\times 200=400$ RTEs computed per iteration, leading to $34800$ RTEs computed in total for GD algorithm whereas only $2094$ RTEs computed for SGD. Same scenario happens in the case when $N=400$ and SGD saves the computation cost significantly.
%}

\begin{table}[h!]	
\centering
\begin{tabular}{ | c | c| c| c | c| c| c |c|}
\hline
\multirow{2}{*}{$N$} & \multicolumn{3}{|c|}{SGD} & \multicolumn{3}{|c|}{GD} &  \multirow{2}{*}{ratio}\\

\cline{2-7}
&\text{RTE per iteration}&\text{ iteration} & \text{total RTEs}&\text{RTE per iteration}&\text{ iteration} & \text{total RTEs} & \\
\hline
100 & 2& 2000& 4000	& 200& 100& 20000 & $20.0\%$\\
\hline
200 & 2&1047 & 2094 & 400& 87& 34800 & $6.02\%$\\
\hline
400 & 2& 935& 1870 & 800& 85& 68000 & $2.75\%$ \\
\hline
\end{tabular}
\caption{Numerical cost comparison: we compare the number of RTEs needs to be computed per iteration, the number of iterations needed for convergence, and the total amout of RTEs requried using SGD and GD. The last column shows the cost ratio. Larger sample size $N$ provides bigger savings.}
\label{tab:cost_SGD}
\end{table}

\subsection{Absorption coefficient recovery}
We recover the absorption coefficient in this subsection following the strategy in Remark~\ref{rmk:absorption}. The scattering coefficient is set as $\sigma_s(x_1,x_2) = 1$ and the to-be-recovered absorption coefficient is set as:
\begin{equation}
\sigma_a(x_1,x_2) = \frac{1}{20}\left[1+8\exp\left(-10(x_1-\frac{1}{4})^2-10(x_2-\frac{1}{4})^2\right)+4\exp\left(-10(x_1-\frac{3}{4})^2-10(x_2-\frac{3}{4})^2\right)\right]\,.
\end{equation}
as plotted in Figure \ref{fig:RealSigma}. 1000 data points $\{(\phi^{(j)},\psi^{(j)}):1\leq j\leq 1000\}$ are prepared. Numerically to run SGD, we set the regularization coefficient $\alpha = 1$, and the learning rate $\eta_n = \frac{\eta_0}{1+\eta_0\alpha n}$ with $\eta_0 = 0.0441$. Two initial guesses are made: one initial guess is a constant away from the true media $\sigma_0 = \sigma + 0.18$, and another being a random initial $\sigma_0 = \sigma R$. The numerical solution after $2000$ iterations are presented in Figure~\ref{fig:Abs_Const} and Figure~\ref{fig:Abs_Ran} for constant deviation and random deviation initial guesses respectively. In Figure \ref{fig:Abs_ReEr} we show the decay of relative errors with respect to the time steps.

\begin{figure*}
	\centering
	
	\begin{subfigure}[t]{0.35\textwidth}
		\centering
		\includegraphics[width=\textwidth]{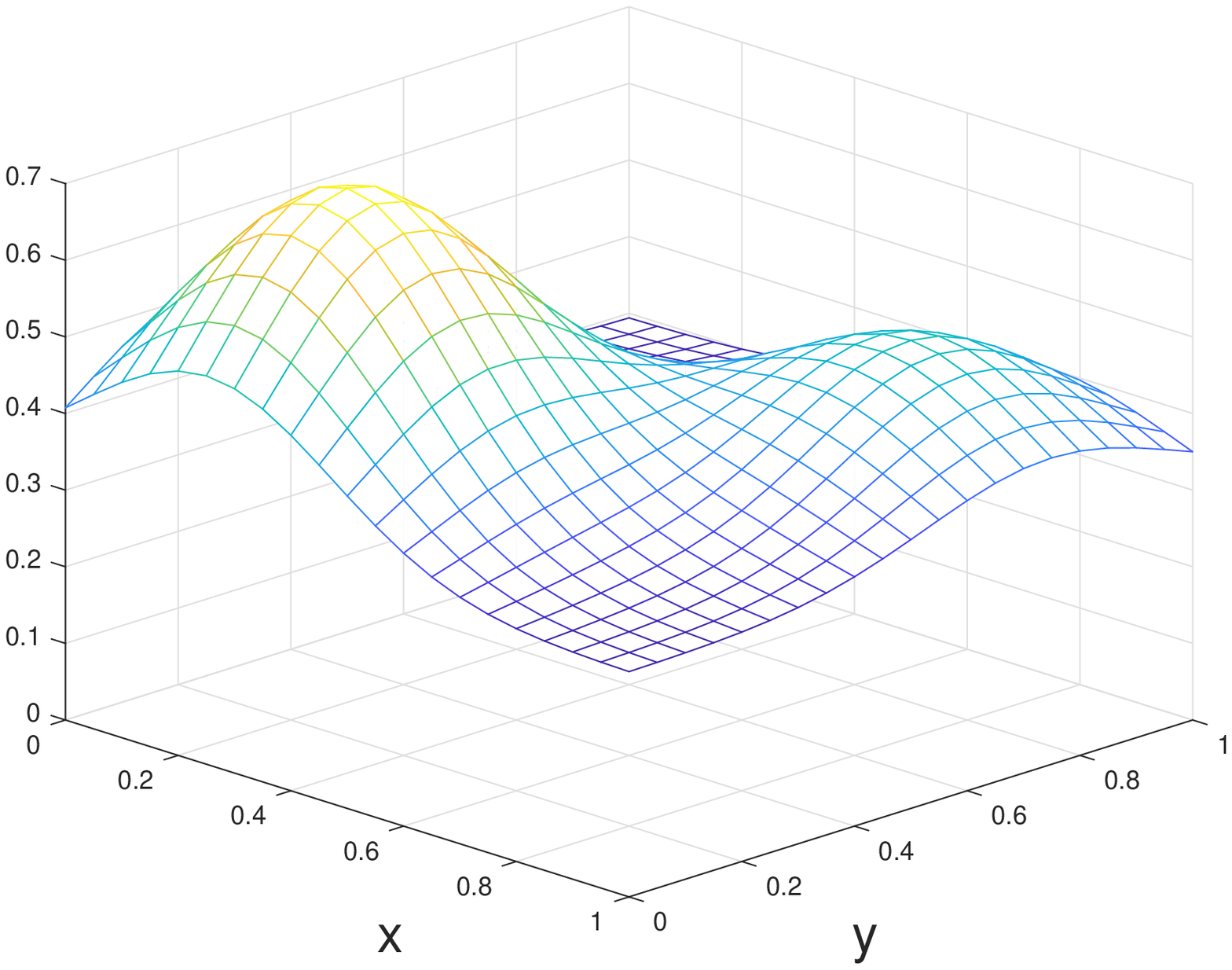}
		\caption{Initial guess $\sigma_0$ is a constant deviation from the true media.}
		\label{fig:Abs_Init_Const}
	\end{subfigure}%
\hspace{0.5in}
	\begin{subfigure}[t]{0.35\textwidth}
		\centering
		\includegraphics[width=\textwidth]{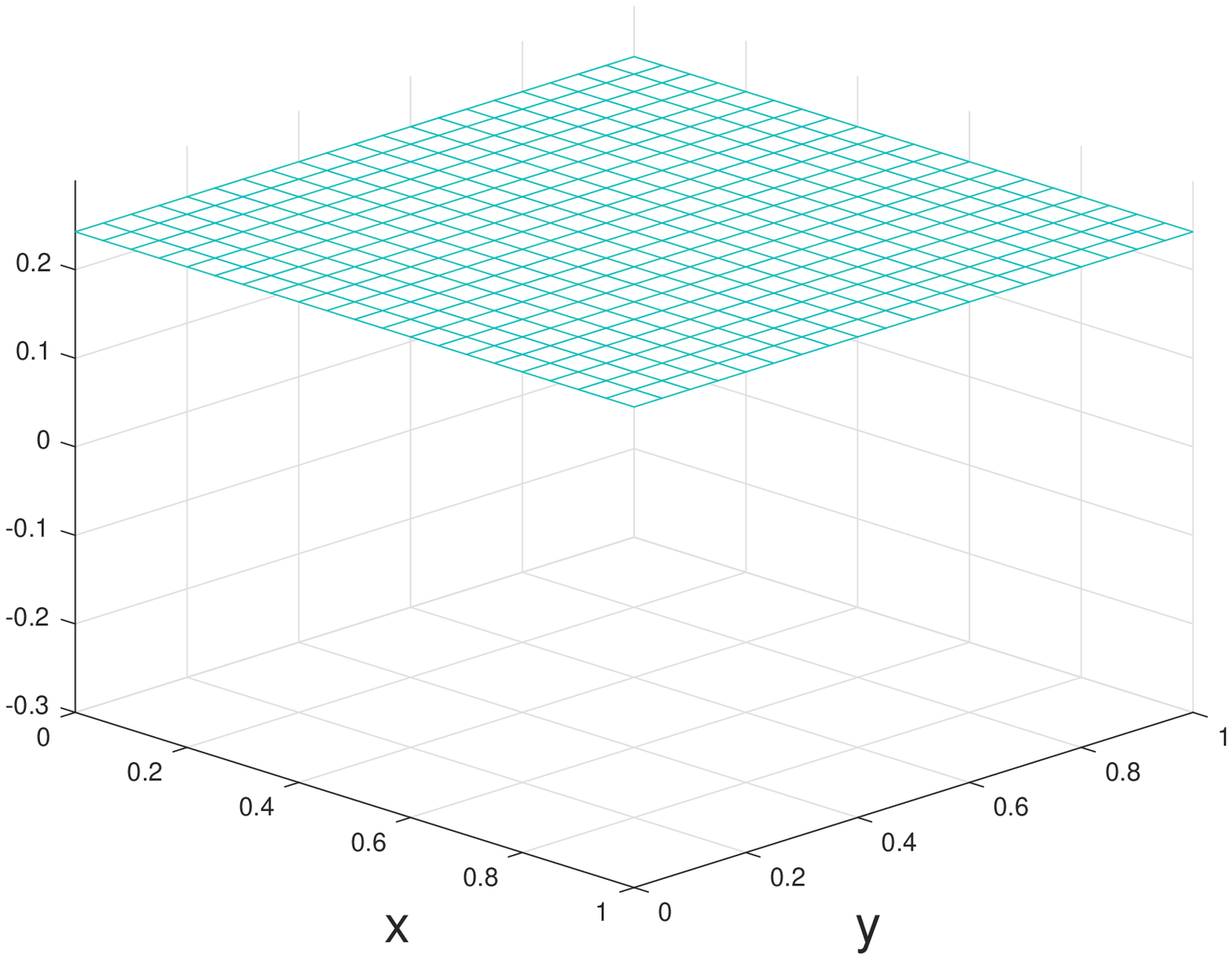}
		\caption{Difference of $\sigma_0$ and the true media.}
		\label{fig:Abs_Dif_Init_Const}
	\end{subfigure}
	
	\bigskip 
	
	\begin{subfigure}[t]{0.35\textwidth}
		\centering
		\includegraphics[width=\textwidth]{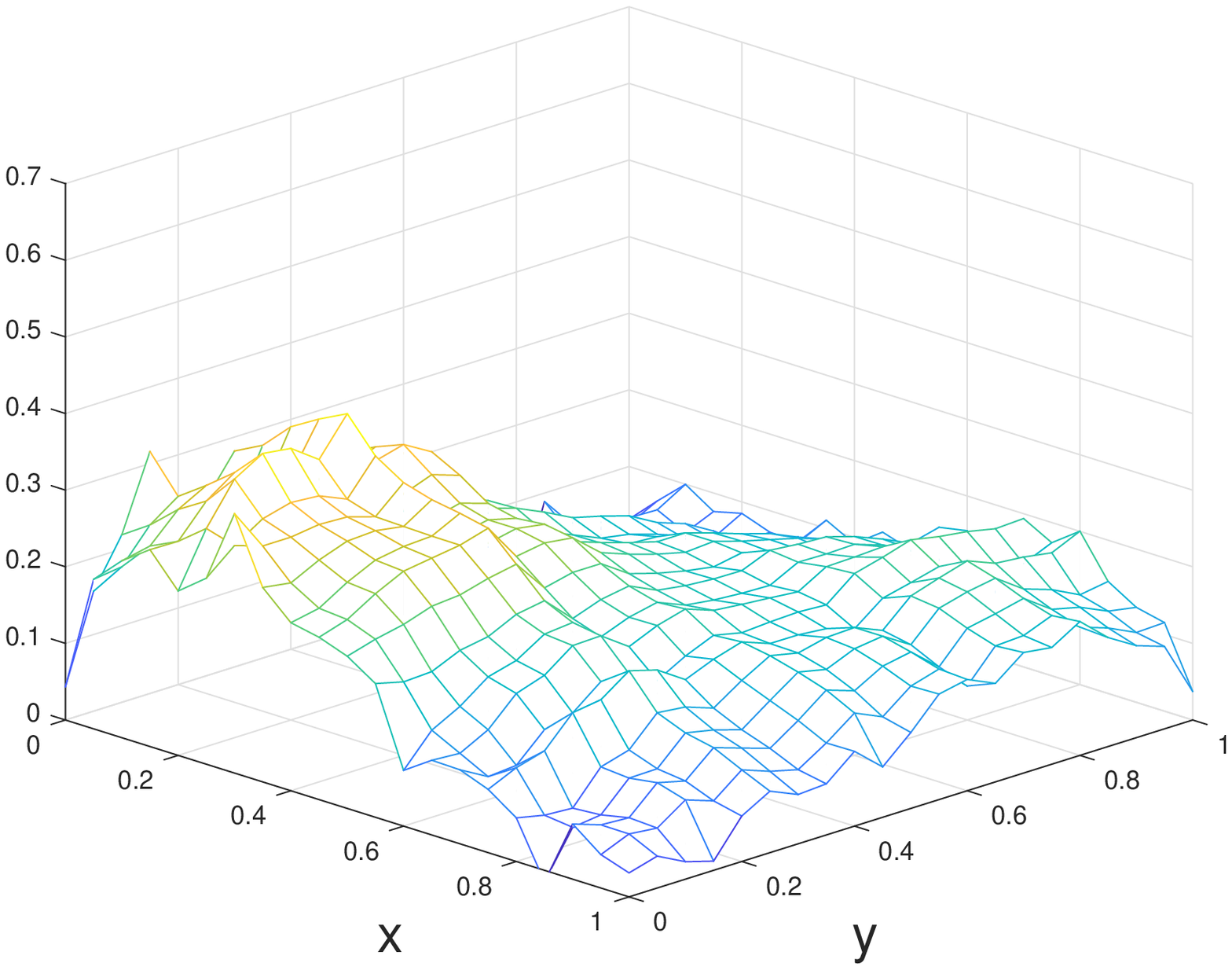}
		\caption{$\sigma_{2000}$.}
		\label{fig:Abs_Final_Const}
	\end{subfigure}%
\hspace{0.5in}
	\begin{subfigure}[t]{0.35\textwidth}
		\centering
		\includegraphics[width=\textwidth]{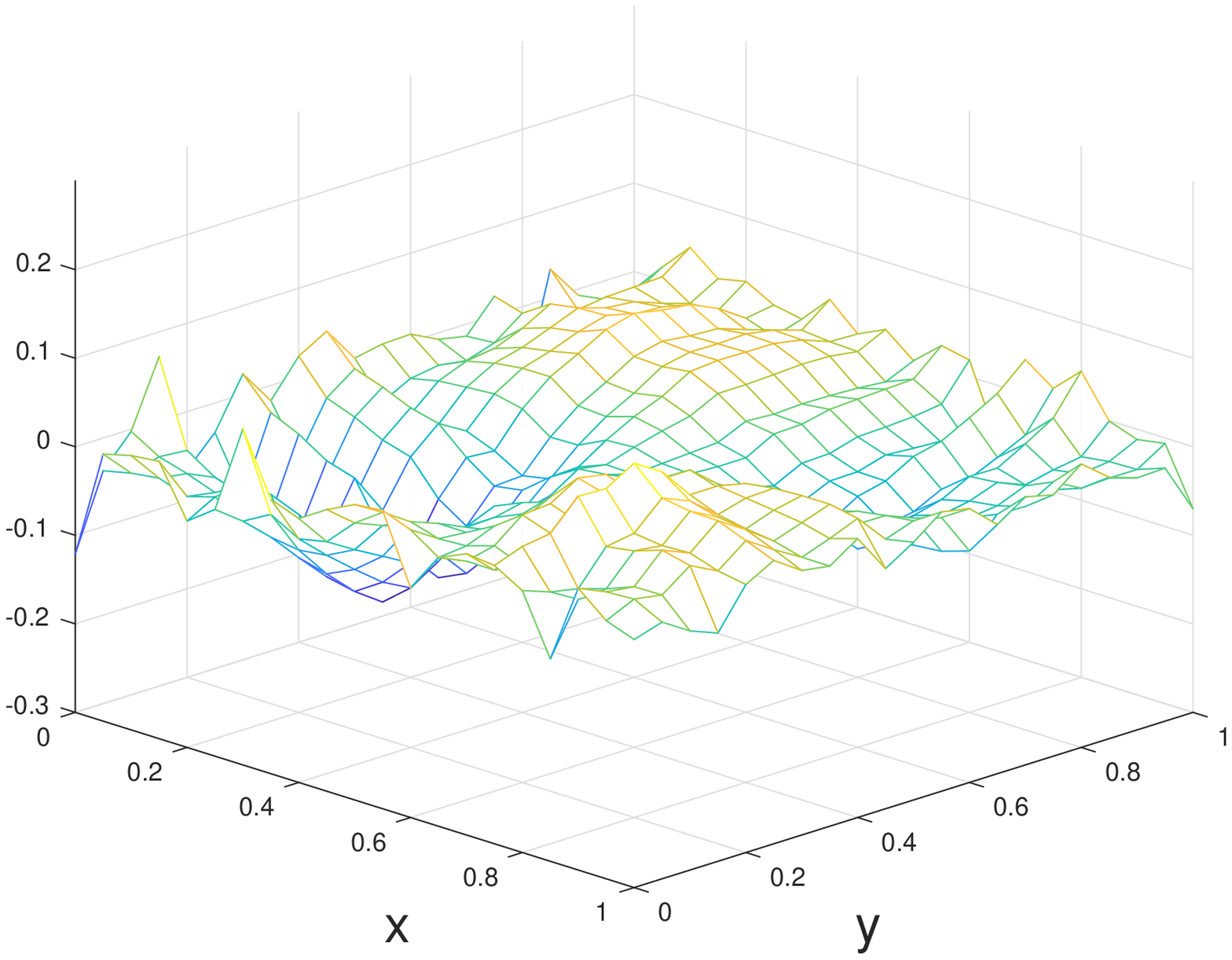}
		\caption{Difference of $\sigma_{2000}$ and the true media.}
		\label{fig:Abs_Dif_Final_Const}
	\end{subfigure}
	
	\caption{The plots show the absorption coeffient recovery. The two plots on the left panel show the media at initial time step and after $2000$ iterations. The errors are shown in the two plots on the right.}
	\label{fig:Abs_Const}
\end{figure*}

\begin{figure*}
	\centering
	
	\begin{subfigure}[t]{0.35\textwidth}
		\centering
		\includegraphics[width=\textwidth]{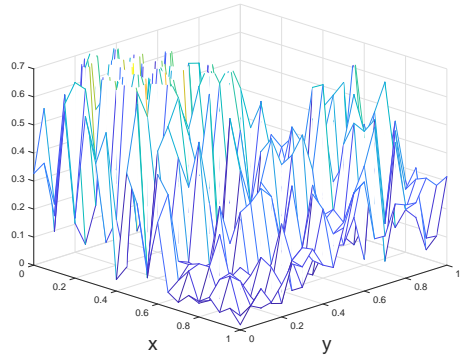}
		\caption{Initial guess $\sigma_0$ is a random field.}
		\label{fig:Abs_Init_Ran}
	\end{subfigure}%
\hspace{0.5in}
	\begin{subfigure}[t]{0.35\textwidth}
		\centering
		\includegraphics[width=\textwidth]{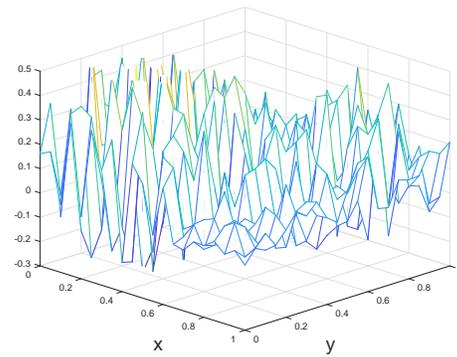}
		\caption{Difference of $\sigma_0$ and the true media.}
		\label{fig:Abs_Dif_Init_Ran}
	\end{subfigure}
	
	\bigskip 
	
	\begin{subfigure}[t]{0.35\textwidth}
		\centering
		\includegraphics[width=\textwidth]{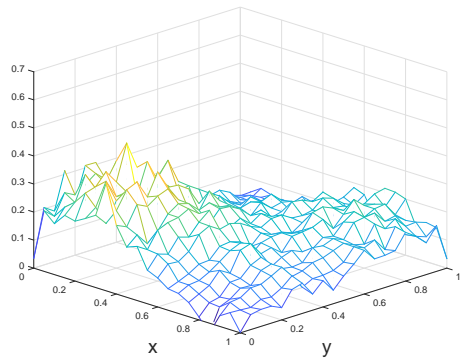}
		\caption{$\sigma_{2000}$.}
		\label{fig:Abs_Final_Ran}
	\end{subfigure}%
\hspace{0.5in}
	\begin{subfigure}[t]{0.35\textwidth}
		\centering
		\includegraphics[width=\textwidth]{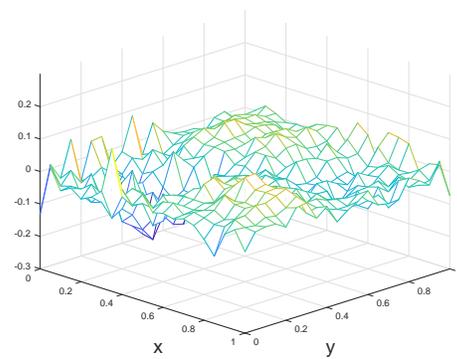}
		\caption{Difference of $\sigma_{2000}$ and the true media.}
		\label{fig:Abs_Dif_Final_Ran}
	\end{subfigure}
	
	\caption{SGD is used to recover the absorption coeffient with initial guess being a random field. The media given at the initial step and after $2000$ iterations are plotted, together with the errors.}
	\label{fig:Abs_Ran}
\end{figure*}

\begin{figure}
	\centering
	\includegraphics[width=0.5\textwidth]{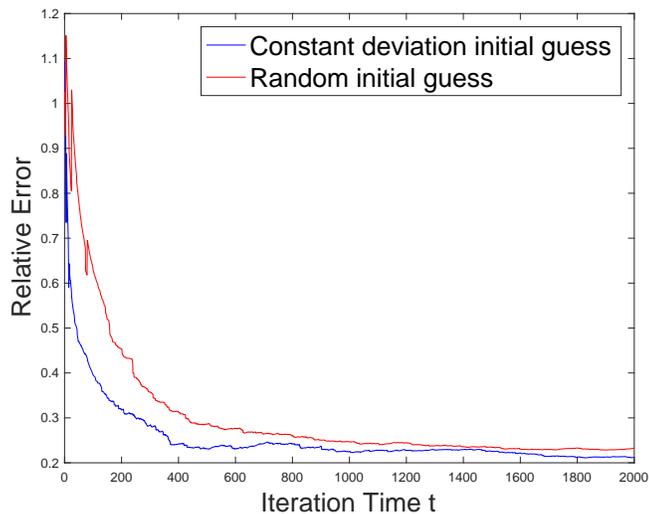}
	\caption{Absorption coefficient recovery. With respect to time steps, the relative error decays almost exponentially in time at the beginning with some flucturation given by the stochastic nature of the algorithm.}
	\label{fig:Abs_ReEr}
\end{figure}

\newpage
\bibliographystyle{siam}
\bibliography{inverse_SGD}

\end{document}